\documentclass[11pt]{article}

\usepackage{fullpage}
\usepackage{amsmath}
\usepackage{amsthm}
\usepackage{dsfont}
\usepackage{amsfonts}
\usepackage{hyperref}
\usepackage{graphicx}
\usepackage{epsfig}
\usepackage[authoryear,round]{natbib}
\usepackage{array}
\usepackage{bbm}
\usepackage{algorithm}
\usepackage{algorithmic}
\usepackage{url}
\usepackage{caption}
\usepackage{subcaption}

\newcommand{\PP}{\mathbb{P}} 
\newcommand{\R}{\mathbb{R}} 
\newcommand{\EE}{\mathbb{E}}      
\newcommand{\var}{\mathrm{Var}} 
\newcommand{\bX}{\mathbf{X}}
\newcommand{\bU}{\mathbf{U}}
\newcommand{\bV}{\mathbf{V}}
\newcommand{\bu}{\mathbf{u}}
\newcommand{\bv}{\mathbf{v}}
\newcommand{\bx}{\mathbf{x}}
\newcommand{\Pd}{\mathcal{P}_d} 
\newcommand{\HH}{\mathcal{H}}
\newcommand{\Sr}{\mathcal{S}}
\newcommand{\F}{\mathcal{F}}
\newcommand{\G}{\mathcal{G}}
\newcommand{\M}{\mathrm{MMD}} 
\newcommand{\HS}{\mathrm{HSIC}} 
\newcommand{\niton}{\not\owns}
\newcommand{\1}{\mathds{1}} 

\newtheorem{definition}{Definition}
\newtheorem{theorem}{Theorem}
\newtheorem{proposition}{Proposition}
\newtheorem{corollary}{Corollary}
\newtheorem{lemma}{Lemma}
\newtheorem{remark}{Remark} 
\newtheorem{assumption}{Assumption}

\date{}
\begin{document}

\title{Kernel-based ANOVA decomposition and Shapley effects -- Application to global sensitivity analysis}

\author{S\'{e}bastien Da Veiga\\
\\
Safran Tech, Modeling \& Simulation\\
Rue des Jeunes Bois, Ch\^{a}teaufort, 78114 Magny-Les-Hameaux, France}

\maketitle

\begin{abstract}
Global sensitivity analysis is the main quantitative technique for identifying the most influential input variables in a numerical simulation model. In particular when the inputs are independent, Sobol' sensitivity indices attribute a portion of the output of interest variance to each input and all possible interactions in the model, thanks to a functional ANOVA decomposition. On the other hand, moment-independent sensitivity indices focus on the impact of input variables on the whole output distribution instead of the variance only, thus providing complementary insight on the inputs / output relationship. Unfortunately they do not enjoy the nice decomposition property of Sobol' indices and are consequently harder to analyze. In this paper, we introduce two moment-independent indices based on kernel-embeddings of probability distributions and show that the RKHS framework used for their definition makes it possible to exhibit a kernel-based ANOVA decomposition. This is the first time such a desirable property is proved for sensitivity indices apart from Sobol' ones. When the inputs are dependent, we also use these new sensitivity indices as building blocks to design kernel-embedding Shapley effects which generalize the traditional variance-based ones used in sensitivity analysis. Several estimation procedures are discussed and illustrated on test cases with various output types such as categorical variables and probability distributions. All these examples show their potential for enhancing traditional sensitivity analysis with a kernel point of view.
\end{abstract}

\section{Introduction}

In the computer experiments community, global sensitivity analysis (GSA) has now emerged as a central tool for exploring the inputs/outputs relationship of a numerical simulation model. Starting from the pioneering work of Sobol \citep{sob93} and  Saltelli \citep{sal99} on the interpretation and estimation of Sobol' indices, the last two decades have been a fertile ground for the development of advanced statistical methodologies and extensions of original Sobol' indices: new estimation procedures (\cite{sdv09}, \cite{sdv13}, \cite{solis19}, \cite{gamboa20}), multivariate outputs with aggregation \citep{gamboa13} and dimensionality reduction \citep{lam10}, goal-oriented sensitivity analysis \citep{fort16} or moment-independent sensitivity measures (\cite{bor07}, \cite{sdv15}), among others. At the heart of the popularity of Sobol' indices is the fundamental functional analysis of variance (ANOVA) decomposition, which opens the path for their interpretation as parts of the output variance and makes it possible to pull apart the input main effects and all their potential interactions, up to their whole influence measured by total Sobol' indices. This decomposition however has two drawbacks. First, it is only valid when the inputs are independent, although some generalizations were investigated \citep{chas12}. Secondly, it only concerns the original Sobol' indices, meaning that it is not possible to split the input effects with goal-oriented or moment-independent sensitivity analysis in general.

When the inputs are dependent, total Sobol' indices can still be used to discriminate them when the objective is to build a surrogate model of the system, and other Sobol'-related indices have also been proposed for interpretability \citep{mara15}. But the major breakthrough happened when Shapley effects have been defined for GSA by Owen \citep{owe14}. Indeed due to their mathematical foundations from game theory, Shapley effects do not require the independence assumption to enjoy nice properties: each input is assigned a Shapley effect lying between $0$ and $1$, while the sum of all effects is equal to $1$. For a given input all interactions and correlations with other ones are mixed up, but the interpretation as parts of the output variance is kept and input rankings are still sensible. For these reasons Shapley effects are now commonly thought as central importance measures in GSA for dealing with dependence, and their estimation has been thoroughly investigated recently \citep{song16,ioopri19,bro20,pli20}.

From an interpretability perspective, other importance measures introduced in the context of goal-oriented and moment-independent sensitivity analysis have proven useful to gain additional insights on a given model. For example quantile-oriented \citep{fort16,maume18} or reliability-based measures \citep{dit96} can help understand which inputs lead to the failure of the system, while optimization-related indices enable dimension reduction for optimization problems \citep{spa19}. On the other hand, moment-independent sensitivity indices, which quantify the input impact on the whole output distribution instead of the variance only, are powerful complementary tools to grasp further types of input influence. Among them are the f-divergence indices (\cite{sdv15}, \cite{rah16}) with particular cases corresponding to the sensitivity index introduced by Borgonovo \citep{bor07} and the class of kernel-based sensitivity indices, which rely on the embedding of probability distributions in reproducing kernel Hilbert spaces (RKHS) \citep{sdv15,sdv16}. Unfortunately an ANOVA-like decomposition is not available yet for any of these indices even in the independent setting: as a consequence this limits the interpretation of their formulation for interactions since without ANOVA it is not possible to remove the main effects, and at the same time the natural normalization constant (equivalent to the total output variance for Sobol' indices) is not known.\\

In this paper we focus on a general RKHS framework for GSA and prove that an ANOVA decomposition actually exists for two previously introduced kernel-based sensitivity indices in the independent setting. To the best of our knowledge this is the first time such a decomposition is available for other sensitivity indices other than the original Sobol' ones. Not only this makes it possible to properly define higher-order indices, but this further gives access to their natural normalization constant. We also demonstrate that these measures are generalizations of Sobol' indices, in the sense that they are recovered with specific kernels. But the RKHS point of view additionally comes with a large body of work on several kernels specifically designed for particular target applications, such as multivariate, functional, categorical or time-series case studies, thus defining a unified framework for many real GSA test cases. When inputs are not independent, we finally introduce a kernel-based version of Shapley effects similar to the ones proposed by Owen.

The paper is organized as follows. Section \ref{sec:gsa} first briefly introduces the traditional functional ANOVA decomposition with Sobol' indices and moment-independent indices. In Section \ref{sec:kernel} we then discuss the elementary tools from RKHS theory needed to build kernel-based sensitivity indices which are at the core of this work. We further investigate these indices and prove they also arise from an ANOVA decomposition. In addition we define Shapley effects with kernels and the benefits of the RKHS framework for GSA are studied through several examples.
Several estimation procedures are then discussed in Section \ref{sec:estim}, where we generalize some of the recent estimators for Sobol' indices. Finally, Section \ref{sec:exp} illustrates the potential of these sensitivity indices with various numerical experiments corresponding to typical GSA applications.

\section{Global sensitivity analysis} \label{sec:gsa}

Let $\eta:\ \mathcal{X}_1\times\ldots\times\mathcal{X}_d \rightarrow \mathcal{Y}$ denote the numerical simulation model, which is a function of $d$ input variables $X_l\in\mathcal{X}_l$, $l=1,\ldots,d$, and $Y\in\mathcal{Y}$ the model output given by $Y=\eta(X_1,\ldots,X_d)$. In standard GSA the inputs $X_l$ are further assumed to be independent with known probability distributions $\textup{P}_{X_l}$, meaning that the vector of inputs $\bX=\left(X_1,\ldots,X_d\right)$ is distributed as $\textup{P}_{\bX}=\textup{P}_{X_1}\otimes\ldots\otimes \textup{P}_{X_d}$. For any subset $A=\{l_1,\ldots,l_{\vert A\vert}\}\in\Pd$ of indices taken from $\{1,\ldots,d\}$ we denote $\bX_A=\left(X_{l_1},\ldots,X_{l_{\vert A\vert}}\right)\in\mathcal{X}_A=\mathcal{X}_{l_1}\times\ldots\times\mathcal{X}_{l_{\vert A\vert}}$ the vector of inputs with indices in $A$ and $\bX_{-A}$ the complementary vector with indices not in $A$. In this setting, the main objective of global sensitivity analysis is to quantify the impact of any group of input variables $\bX_A$ on the model output $Y$.\\
In this section we first recall the functional ANOVA decomposition and the definition of Sobol' indices, which fall into the category of \textit{variance-based} indices. Sensitivity indices that account for the whole output distribution, referred to as \textit{moment-independent} indices, are then discussed. Note that in the following, we adopt the notation $S$ for a properly normalized sensitivity index, while $\Sr$ will stand for an unnormalized index, where normalization is to be understood as an end result from an ANOVA-like decomposition.

\subsection{ANOVA decomposition and variance-based sensitivity indices}

Here we first assume that $Y\in\mathcal{Y}\subset\mathbb{R}$ is a square integrable scalar output. If the inputs are independent, the function $\eta$ can then be decomposed according to the ANOVA decomposition:

\begin{theorem}[ANOVA decomposition \citep{hoe48,ant84}]
\label{th:anova}
Assume that $\eta:\ \mathcal{X}_1\times\ldots\times\mathcal{X}_d \rightarrow \mathcal{Y}$ is a square integrable function of $d$ independent random variables $X_1,\ldots,X_d$. Then $\eta$ admits a decomposition
\begin{equation*}
Y = \eta(X_1,\ldots,X_d) = \sum_{A \subseteq \Pd} \eta_A(\bX_A),
\end{equation*}
with $\eta_A$ depending only on the variables $\bX_A$ and satisfying
\begin{itemize}
\item[(a)] $\eta_\emptyset = \EE(Y)$,
\item[(b)]  $\EE_{X_l}(\eta_A(\bX_A))=0$ if $l\in A$,
\item[(c)]  $\eta_A(\bX_A) = \sum_{B \subset A} (-1)^{\vert A\vert - \vert B\vert} \EE(Y\vert \bX_B)$.
\end{itemize}
Furthermore, $(b)$ implies that all the terms $\eta_A$ in the decomposition are mutually orthogonal. As a consequence, the output variance can be decomposed as
\begin{equation}
\var\, Y = \sum_{A \subseteq \Pd} \var\, \eta_A(\bX_A) = \sum_{A \subseteq \Pd} V_A  \label{eq:anova}
\end{equation}
where
\begin{equation}
V_A = \sum_{B \subset A} (-1)^{\vert A\vert - \vert B\vert} \var\, \EE(Y\vert \bX_B).  \label{eq:va}
\end{equation}
\end{theorem}

When this decomposition holds, it is then straightforward to quantify the influence of any subset of inputs $\bX_A$ on the output variance by normalizing each term with $\var\, Y$.
\begin{definition}[Sobol' indices \citep{sob93}]
Under the same assumptions of Theorem \ref{th:anova}, the Sobol' sensitivity index associated to a subset $A$ of input variables is defined as
\begin{equation}
S_A = \frac{V_A}{\var\, Y}, \label{eq:SA}
\end{equation}
while the total Sobol' index associated to $A$ is
\begin{equation}
S^T_A = \sum_{B\subseteq \Pd,\, B\cap A \neq \emptyset}S_B. \label{eq:STA}
\end{equation}
In particular, the first-order Sobol' index of an input $X_l$ writes
\begin{equation*}
S_l = \frac{\var\, \EE(Y\vert X_l)}{\var\, Y}
\end{equation*}
and its total Sobol' index is given by
\begin{equation*}
S^T_l = \sum_{B\subseteq \Pd,\, l\in B}S_B = 1 - \frac{\var\, \EE(Y\vert \bX_{-l})}{\var\, Y}.
\end{equation*}
Finally, the ANOVA decomposition (\ref{eq:anova}) readily provides an interpretation of Sobol' indices as a percentage of explained output variance, \textit{i.e.}
\begin{equation}
\sum_{A \subseteq \Pd} S_A = 1. \label{eq:sumSA}
\end{equation}
\end{definition}

With these definitions, the impact of each input variable can be quantitatively assessed: the first-order Sobol' index measures the main effect of an input, while the total Sobol' index aggregates all its potential interactions with other inputs. As an illustration, an input variable with low total Sobol' index is thus unimportant and one can freeze it at a default value. When for a given input both first-order and total Sobol' indices are close, this means that this input does not have interactions, while a large gap indicates strong interactions in the model. Furthermore, due to the summation property (\ref{eq:sumSA}), the interpretation of Sobol' indices as shares of the output variance is an efficient tool for practitioners who aim at understanding precisely the impact and interactions of the inputs of a model on the output. For example the interaction of two inputs $X_l$ and $X_{l'}$ writes
\begin{equation}
S_{ll'} = \frac{\var\, \EE(Y\vert X_l,\, X_{l'}) - \var\, \EE(Y\vert X_l) - \var\, \EE(Y\vert X_{l'})}{\var\, Y} = \frac{\var\, \EE(Y\vert X_l,\, X_{l'})}{\var\, Y} - S_l - S_{l'}. \label{eq:inter}
\end{equation}
Note that to compute this interaction one subtracts the first-order indices $S_l$ and $S_{l'}$ from the sensitivity index of the subset $\left(X_l,X_{l'}\right)$ in order to remove the main effects and highlight the interaction only.

\subsection{Moment-independent sensitivity indices}

Despite their appealing properties, Sobol' indices rank the input variables according to their impact on the output variance only. In a parallel line of work, several authors proposed to investigate instead how inputs influence the whole output distribution, thus introducing a different insight on the inputs/outputs relationship. The starting point (\cite{baucells13}, \cite{sdv15}) is to consider that a given input $X_l$ is important in the model if the probability distribution $\textup{P}_Y$ of the output changes when $X_l$ is fixed, \textit{i.e.} if the conditional probability distribution $\textup{P}_{Y\vert X_l}$ is different from $\textup{P}_Y$. More precisely, if $d(\cdot,\cdot)$ denotes a dissimilarity measure between probability distributions, one can define a sensitivity index for variable $X_l$ given by
\begin{equation}
\Sr_l = \EE_{X_l}\left(d(\textup{P}_Y,\textup{P}_{Y\vert X_l})\right). \label{eq:dis}
\end{equation}
Such a general formulation is flexible, in the sense that many choices for $d(\cdot,\cdot)$ are available. As an illustration, it is straightforward to show that the unnormalized first-order Sobol' index is retrieved with the naive dissimilarity measure $d(\textup{P},\textup{Q})=\left(\EE_{\xi\sim \textup{P}}(\xi)-\EE_{\xi\sim \textup{Q}}(\xi)\right)^2$, which compares probability distributions only through their means. A large class of dissimilarity measures is also given by the so-called \textit{f-divergence} family: assuming that $(X_l,Y)$ has an absolute continuous distribution with respect to the Lebesgue measure on $\R^2$, the f-divergence between $\textup{P}_Y$ and $\textup{P}_{Y\vert X_l}$ is
\begin{equation*}
d_f(\textup{P}_Y,\textup{P}_{Y\vert X_l}) = \int f\left(\frac{p_Y(y)}{p_{Y\vert X_l}(y)}\right)p_{Y\vert X_l}(y)dy
\end{equation*}
where $f$ is a convex function such that $f(1)=0$ and $p_Y$ and $p_{Y\vert X_l}$ are the probability distribution functions of $Y$ and $Y\vert X_l$, respectively. The corresponding sensitivity index is then 
\begin{equation*}
\Sr^f_l = \int  f\left(\frac{p_Y(y)p_{X_l}(x)}{p_{X_l,Y}(x,y)}\right)p_{X_l,Y}(x,y)dxdy
\end{equation*}
with $p_{X_l}$ and $p_{X_l,Y}$ the probability distribution functions of $X_l$ and $(X_l,Y)$, respectively. This index has been studied for example in \cite{sdv15} and \cite{rah16}. A notable special case is obtained with the total-variation distance corresponding to $f(t)=\vert t-1\vert$, leading to the sensitivity index proposed by Borgonovo \citep{bor07}:
\begin{equation*}
\Sr^{TV}_l = \int \vert p_Y(y)p_{X_l}(x) - p_{X_l,Y}(x,y)\vert dxdy.
\end{equation*}
Obviously, definition (\ref{eq:dis}) can be easily extended to measure the influence of any subset of inputs $\Sr_A = \EE_{\bX_A}\left(d(P_Y,P_{Y\vert \bX_A})\right)$. But in this case, since there is no ANOVA-like decomposition, there is no longer the guarantee that an interaction index defined following (\ref{eq:inter}):
\begin{equation*}
\Sr^{TV}_{ll'} = \int \vert p_Y(y)p_{X_l}(x)p_{X_{l'}}(x') - p_{X_l,X_{l'},Y}(x,x',y)\vert dxdx'dy - \Sr^{TV}_l - \Sr^{TV}_{l'}
\end{equation*}
really measures the pure interaction between $X_l$ and $X_{l'}$. Therefore the interpretation of higher-order moment-independent sensitivity indices is cumbersome. On the other hand, even if normalization constants have been proposed through general inequalities on f-divergences (\cite{bor07}, \cite{rah16}), the lack of an ANOVA decomposition once again impedes the definition of a natural normalization constant equivalent to the output variance for Sobol' indices.\\

Recently, new moment-independent indices built upon the framework of RKHS embedding of probability distributions have also been investigated \citep{sdv15,sdv16}. Though originally introduced as an alternative to reduce the curse of dimensionality and make the most of the vast kernel literature, we will see in what follows that they actually exhibit an ANOVA-like decomposition and can therefore be seen as a general kernelized version of Sobol' indices.

\section{Kernel-based sensitivity analysis} \label{sec:kernel}

Before introducing the kernel-based sensitivity indices, we first review some elements of the RKHS embedding of probability distributions \citep{smola07}, which will serve as a building block for their definition.

\subsection{RKHS embedding of distributions}

We first introduce a RKHS $\HH$ of functions $\mathcal{X}\rightarrow \R$ with kernel $k_{\mathcal{X}}$ and dot product $\langle\cdot,\cdot\rangle_{\HH}$.
The \textit{kernel mean embedding} $\mu_{\textup{P}}\in\HH$ of a probability distribution $\textup{P}$ on $\mathcal{X}$ is defined as
\begin{equation*}
\mu_{\textup{P}} = \EE_{\xi\sim \textup{P}}k_{\mathcal{X}}(\xi,\cdot) = \int_{\mathcal{X}} k_{\mathcal{X}}(\xi,\cdot) d\textup{P}(\xi)
\end{equation*}
if $\EE_{\xi\sim \textup{P}}k_{\mathcal{X}}(\xi,\xi)<\infty$, see \cite{smola07}. The representation $\mu_{\textup{P}}$ is appealing because, if the kernel $k_{\mathcal{X}}$ is characteristic, the map $\textup{P}\rightarrow \mu_{\textup{P}}$ is injective \citep{sri09,sri10}. Consequently, the kernel mean embedding can be used in lieu of the probability distribution for several comparisons and manipulations of probability measures but using only inner products or distances in the RKHS.
For example, a distance between two probability measures $\textup{P}_1$ and $\textup{P}_2$ on $\mathcal{X}$ can simply be obtained by computing the distance between their representations in $\HH$, \textit{i.e.}
\begin{equation*}
\M(\textup{P}_1,\textup{P}_2) = \Vert \mu_{\textup{P}_1} - \mu_{\textup{P}_2} \Vert_{\HH},
\end{equation*}
which is a distance if the kernel $k_{\mathcal{X}}$ is characteristic \citep{sri09,sri10}. This distance is called the \textit{maximum mean discrepancy} (MMD) and it has been recently used in many applications \citep{mua12,szabo16}. Indeed, using the reproducing property of a RKHS one may show \citep{song08} that
\begin{equation*}
\M^2(\textup{P}_1,\textup{P}_2) = \EE_{\xi,\xi'}k_{\mathcal{X}}(\xi,\xi') - 2 \EE_{\xi,\zeta}k_{\mathcal{X}}(\xi,\zeta) + \EE_{\zeta,\zeta'}k_{\mathcal{X}}(\zeta,\zeta') 
\end{equation*}
where $\xi,\xi'\sim \textup{P}_1$ and $\zeta,\zeta'\sim \textup{P}_2$ with $\xi,\ \xi',\ \zeta,\ \zeta'$ independent, this notation being used throughout the rest of the paper. This means that the MMD can be computed with expectations of kernels only, unlike other distances between probability distributions which will typically require density estimation.\\

Another significant application of kernel embeddings concerns the problem of measuring the dependence between random variables. Given a pair of random vectors $(\bU,\bV)\in\mathcal{X}\times\mathcal{Y}$ with probability distribution $\textup{P}_{\bU\bV}$, we define the product RKHS $\HH=\F\times\G$ with kernel $k_H((\bu,\bv),(\bu',\bv'))=k_\mathcal{X}(\bu,\bu')k_\mathcal{Y}(\bv,\bv')$. A measure of the dependence between $\bU$ and $\bV$ can then be defined as the distance between the mean embedding of $\textup{P}_{\bU\bV}$ and $\textup{P}_{\bU}\otimes \textup{P}_{\bV}$, the joint distribution with independent marginals $\textup{P}_{\bU}$ and $\textup{P}_{\bV}$:
\begin{equation*}
\M^2(\textup{P}_{\bU\bV},\textup{P}_{\bU}\otimes \textup{P}_{\bV}) = \Vert \mu_{\textup{P}_{\bU\bV}} - \mu_{\textup{P}_{\bU}}\otimes \mu_{\textup{P}_{\bV}}\Vert_{\HH}.
\end{equation*}
This measure is the so-called \textit{Hilbert-Schmidt independence criterion} (HSIC, see \cite{gretton05a,gretton05b}) and can be expanded as
\begin{eqnarray}
\HS(\bU,\bV) &=& \M^2(\textup{P}_{\bU\bV},\textup{P}_{\bU}\otimes \textup{P}_{\bV})\nonumber \\
 &=& \EE_{\bU,\bU',\bV,\bV'} k_\mathcal{X}(\bU,\bU')k_\mathcal{Y}(\bV,\bV')\nonumber \\
&+& \EE_{\bU,\bU'} k_\mathcal{X}(\bU,\bU')\EE_{\bV,\bV'} k_\mathcal{Y}(\bV,\bV')\nonumber \\
&-& 2 \EE_{\bU,\bV}\left[\EE_{\bU'} k_\mathcal{X}(\bU,\bU')\EE_{\bV'} k_\mathcal{Y}(\bV,\bV')\right] \label{eq:hsic}
\end{eqnarray}
where $(\bU',\bV')$ is an independent copy of $(\bU,\bV)$. Once again, the reproducing property implies that HSIC can be expressed as expectations of kernels, which facilitates its estimation when compared to other dependence measures such as the mutual information.

\subsection{Kernel-based ANOVA decomposition}

The RKHS framework introduced above can readily be used to define kernel-based sensitivity indices. The first approach relies on the MMD, while the second one builds upon HSIC. We discuss them below and show that in particular both of them admit an ANOVA-like decomposition.

\subsubsection{MMD-based sensitivity index}
The first natural idea is to come back to the general formulation for moment-independent indices (\ref{eq:dis}) and use the MMD as the dissimilarity measure to compare $\textup{P}_Y$ and $\textup{P}_{Y\vert X_l}$ as proposed in \cite{sdv16}:
\begin{eqnarray*}
\Sr_l^{\M} &=& \EE_{X_l} \M^2(\textup{P}_Y,\textup{P}_{Y\vert X_l})\\
&=& \EE_{X_l} \EE_{\xi,\xi'\sim \textup{P}_Y}k_{\mathcal{Y}}(\xi,\xi') - 2 \EE_{X_l} \EE_{\xi\sim \textup{P}_Y,\zeta\sim \textup{P}_{Y\vert X_l}}k_{\mathcal{Y}}(\xi,\zeta) + \EE_{X_l}\EE_{\zeta,\zeta'\sim \textup{P}_{Y\vert X_l}}k_{\mathcal{Y}}(\zeta,\zeta') \\
&=& \EE_{X_l} \EE_{\zeta,\zeta'\sim \textup{P}_{Y\vert X_l}}k_{\mathcal{Y}}(\zeta,\zeta') - \EE_{\xi,\xi'\sim \textup{P}_Y}k_{\mathcal{Y}}(\xi,\xi')
\end{eqnarray*}
where we have defined a RKHS $\G$ of functions $\mathcal{Y}\rightarrow \R$ with kernel $k_{\mathcal{Y}}$.
More generally, we can also consider the unnormalized MMD-based sensitivity index for a group of variables $\bX_A$ given by $\EE_{\bX_A}\left(\textup{MMD}^2(\textup{P}_Y,\textup{P}_{Y \vert \bX_A})\right)=\EE_{\bX_A} \EE_{\zeta,\zeta'\sim \textup{P}_{Y\vert \bX_A}}k_{\mathcal{Y}}(\zeta,\zeta') - \EE_{\xi,\xi'\sim \textup{P}_Y}k_{\mathcal{Y}}(\xi,\xi')$, provided the following assumption holds:
\begin{assumption}
\label{as:finitekernely}
$\forall A\subseteq\Pd$ and $\forall\bx_A\in\mathcal{X}_A$, $\EE_{\xi\sim\textup{P}_{Y\vert\bX_A=\bx_A}}k_{\mathcal{Y}}(\xi,\xi)<\infty$ with the convention $\textup{P}_{Y\vert\bX_A}=\textup{P}_{Y}$ if $A=\emptyset$.
\end{assumption}

First note that if we focus on the scalar output case $\mathcal{Y}\subset\R$ with the linear kernel $k_{\mathcal{Y}}(y,y')=yy'$, we have
\begin{eqnarray*}
& \Sr_{A}^{\M} &= \EE_{\bX_A}\left(\EE_{\xi\sim \textup{P}_Y}(\xi) - \EE_{\zeta\sim \textup{P}_{Y \vert \bX_A}}(\zeta)\right)^2 \\
& &= \EE_{\bX_A}\left(\EE Y - \EE(Y \vert \bX_A)\right)^2 \\
& &= \var\, \EE(Y \vert \bX_A),
\end{eqnarray*}
that is, we recover the unnormalized Sobol' index for $\bX_A$. $\Sr_A^{\M}$ can thus be seen as a kernelized version of Sobol' indices since the latter can be retrieved with a specific kernel. However it is obvious that since the linear kernel is not characteristic, the MMD in this case is not a distance, which means that $\Sr_A^{\M}$ is no longer a moment-independent index.\\
To make another connection with Sobol' indices, we now recall Mercer's theorem, a notable representation theorem for kernels.
 \begin{theorem}[Mercer, see \cite{aub00}]
Suppose $k_{\mathcal{Y}}$ is a continuous symmetric positive definite kernel on a compact set $\mathcal{Y}$ and consider the integral operator $T_{k_{\mathcal{Y}}}:\ \mathbb{L}^2(\mathcal{Y}) \rightarrow \mathbb{L}^2(\mathcal{Y})$ defined by
\begin{equation*}
\left(T_{k_{\mathcal{Y}}}f\right)(x) = \int_{\mathcal{Y}} k_{\mathcal{Y}}(y,u)f(u)du.
\end{equation*}
Then there is an orthonormal basis $\{e_r\}$ of $\mathbb{L}^2(\mathcal{Y})$ consisting of eigenfunctions of $T_{k_{\mathcal{Y}}}$ such that the corresponding sequence of eigenvalues $\{\lambda_r\}$ are non-negative. The eigenfunctions corresponding to non-zero eigenvalues are continuous on $\mathcal{Y}$ and $k_{\mathcal{Y}}$ has the following representation
\begin{equation*}
k_{\mathcal{Y}}(y,y') = \sum_{r=1}^{\infty} \lambda_r e_r(y) e_r(y')
\end{equation*}
where the convergence is absolute and uniform.
\end{theorem}\label{th:mercerrkhs}
Assume now that the output $Y\in\mathcal{Y}$ with $\mathcal{Y}$ a compact set, meaning that Mercer's theorem holds. Then $k_{\mathcal{Y}}$ admits a representation 
\begin{equation*}
k_{\mathcal{Y}}(y,y') = \sum_{r=1}^{\infty} \phi_r(y)\phi_r(y')
\end{equation*}
where $\phi_r(y)=\sqrt{\lambda_r} e_r(y)$ are orthogonal functions in $\mathbb{L}^2(\mathcal{Y})$. In this setting we can write
\begin{eqnarray*}
 \Sr_{A}^{\M} = \EE_{{\bX}_A}\left(\M^2(\textup{P}_{Y},\textup{P}_{Y \vert {\bX}_A})\right) &=& \EE_{{\bX}_A} \EE_{\xi,\xi'\sim\textup{P}_{Y \vert {\bX}_A}} k_{\mathcal{Y}}(\xi,\xi') - \EE_{\zeta,\zeta'\sim\textup{P}} k_{\mathcal{Y}}(\zeta,\zeta') \\
 &=& \EE_{{\bX}_A} \EE_{\xi,\xi'\sim\textup{P}_{Y | {\bX}_A}}\left(\sum_{r=1}^{\infty} \phi_r(\xi)\phi_r(\xi')\right)\\
 & & - \EE_{\zeta,\zeta'\sim\textup{P}}\left(\sum_{r=1}^{\infty} \phi_r(\zeta)\phi_r(\zeta')\right).
\end{eqnarray*}
Now, since the convergence of the series is absolute, we can interchange the expectations and the summations to get
\begin{eqnarray}
\Sr_{A}^{\M} &=& \sum_{r=1}^{\infty} \bigg\{\EE_{{\bX}_A} \EE_{\xi,\xi'\sim\textup{P}_{Y \vert {\bX}_A}}\left(\phi_r(\xi)\phi_r(\xi')\right)  - \EE_{\zeta,\zeta'\sim\textup{P}}\left(\phi_r(\zeta)\phi_r(\zeta')\right)\bigg\}\nonumber \\
 &=& \sum_{r=1}^{\infty} \bigg\{\EE_{{\bX}_A}\EE\left(\phi_r(Y)\vert{\bX}_A\right)^2 - \EE\left(\phi_r(Y)\right)^2\bigg\}\nonumber \\
 &=& \sum_{r=1}^{\infty} \var\, \EE\left(\phi_r(Y)|{\bX}_A\right). \label{eq:MMDindexgroup}
\end{eqnarray}

In other words, the MMD-based sensitivity index $\Sr_{A}^{\M}$ generalizes the Sobol' one in the sense that it measures the impact of the inputs not only on the conditional expectation of the output, but on a possibly infinite number of transformations $\phi_r$ of the output, given by the eigenfunctions of the kernel.\\

We can now state the main theorem of this section on the ANOVA-like decomposition for $\Sr_{A}^{\M}$. Recall that the variance decomposition (\ref{th:anova}) states that the variance of the output can be decomposed as $\var\, Y = \sum_{A\subseteq\Pd} V_A$ where each term is given by
\begin{equation*}
V_A = \sum_{B \subset A}(-1)^{|A|-|B|}\var\, \EE ( Y|{\bX}_B).
\end{equation*}
The MMD-based equivalent is obtained with the following theorem.
\begin{theorem}[ANOVA decomposition for MMD]
\label{th:MMDanova}
Under the same assumptions of Theorem \ref{th:anova} (in particular, the random vector $\bX$ has independent components) and with Assumption \ref{as:finitekernely}, denote $\M^2_{\textup{tot}}=\EE k_{\mathcal{Y}}(Y,Y)-\EE k_{\mathcal{Y}}(Y,Y')$ where $Y'$ is an independent copy of $Y$. Then the total MMD can be decomposed as
\begin{equation*}
\M^2_{\textup{tot}} = \sum_{A\subseteq\Pd} \M^2_A
\end{equation*}
where each term is given by
\begin{equation*}
\M^2_A = \sum_{B \subset A}(-1)^{|A|-|B|}\EE_{{\bX}_B}\left(\M^2(\textup{P}_{Y},\textup{P}_{Y \vert {\bX}_B})\right) .
\end{equation*}
\end{theorem}

The proof is given in Appendix \ref{sec:MMDanova}. Theorem \ref{th:MMDanova} is very similar to the ANOVA one given in (\ref{th:anova}):
one can note that the total variance of the output is replaced by a generalized variance $\M^2_{\textup{tot}}$ defined by the kernel, and that each subset effect is obtained by removing lesser order ones in the MMD distance of the conditional distributions (instead of the variance of the conditional expectations in the ANOVA).
The following corollary states that these two decompositions coincide when the kernel is chosen as the linear one.
\begin{corollary}
\label{cor:MMDsobol}
When $Y\in\mathcal{Y}\subset\R$ and $k_{\mathcal{Y}}(y,y')=yy'$ in Theorem \ref{th:MMDanova}, the decomposition is identical to the decomposition (\ref{th:anova}), which means that
\[\M^2_{\textup{tot}} = \var\; Y \mbox{ and } \forall B\in\Pd,\ \EE_{{\bX}_B}\left(\M^2(\textup{P}_{Y},\textup{P}_{Y \vert {\bX}_B})\right) = \var\, \EE (Y\vert{\bX}_B).\]
It further implies $\forall A\subseteq\Pd,\ \M^2_A = V_A$.
\end{corollary}
Thanks to Theorem \ref{th:MMDanova} we can now define properly normalized MMD-based indices.
\begin{definition}[MMD-based sensitivity indices]
\label{defi:normMMDindex}
In the frame of Theorem \ref{th:MMDanova}, let $A\subseteq\Pd$. The normalized MMD-based sensitivity index associated to a subset $A$ of input variables is defined as
\begin{equation*}
S_A^{\M} = \frac{\M^2_A}{\M^2_{\textup{tot}}},
\end{equation*}
while the total MMD-based index associated to $A$ is 
\begin{equation*}
S_A^{\mbox{\scriptsize{T}},\M} = \sum_{B\subseteq \Pd,\, B\cap A \neq \emptyset}S_B^{\M} = 1 - \frac{\EE_{{\bX}_{-A}}\left(\M^2(\textup{P}_{Y},\textup{P}_{Y \vert {\bX}_{-A}})\right)}{\M^2_{\textup{tot}}}.
\end{equation*}
From Theorem \ref{th:MMDanova}, we have the fundamental identity providing the interpretation of MMD-based indices as percentage of the explained generalized variance $\M^2_{\textup{tot}}$:
\begin{equation*}
\sum_{A\subseteq\Pd}S_A^{\M}=1.
\end{equation*}
\end{definition}
Finally, we exhibit a generalized law of total variance for $\M^2_{\textup{tot}}$ which will yield another formulation for the total MMD-based index.
\begin{proposition}[Generalized law of total variance]
\label{prop:lawtotalvar}
Assuming Assumption \ref{as:finitekernely} holds, we have
\begin{equation*}
\M^2_{\textup{tot}} = \EE_{{\bX}_A}\left[\EE_{\xi\sim\textup{P}_{Y \vert {\bX}_A}} k_{\mathcal{Y}}(\xi,\xi) - \EE_{\xi,\xi'\sim\textup{P}_{Y \vert {\bX}_A}} k_{\mathcal{Y}}(\xi,\xi')\right] + \EE_{{\bX}_A}\left(\M^2(\textup{P}_{Y},\textup{P}_{Y \vert {\bX}_A})\right).
\end{equation*}
\end{proposition}
The proof is to be found in Appendix \ref{sec:lawtotalvar}. This is a generalization in the sense that the total variance is replaced by $\M^2_{\textup{tot}}$, the conditional variance by $\EE_{\xi\sim\textup{P}_{Y \vert {\bX}_A}} k_{\mathcal{Y}}(\xi,\xi) - \EE_{\xi,\xi'\sim\textup{P}_{Y \vert {\bX}_A}} k_{\mathcal{Y}}(\xi,\xi')$ and the variance of the conditional expectation by $\EE_{{\bX}_A}\left(\M^2(\textup{P}_{Y},\textup{P}_{Y \vert {\bX}_A})\right)$. In particular, all these terms reduce to the ones in the classical law of total variance if one uses the linear kernel $k_{\mathcal{Y}}(y,y')=yy'$ in the scalar case. This gives the following corollary.
\begin{corollary}[Other formulation of total MMD-based index]
In the frame of Theorem \ref{th:MMDanova}, we have for all $A\subseteq\Pd$
\begin{equation*}
S_A^{\mbox{\scriptsize{T}},\M} = \frac{\EE_{{\bX}_{-A}}\left[\EE_{\xi\sim\textup{P}_{Y \vert {\bX}_{-A}}} k_{\mathcal{Y}}(\xi,\xi) - \EE_{\xi,\xi'\sim\textup{P}_{Y \vert {\bX}_{-A}}} k_{\mathcal{Y}}(\xi,\xi')\right]}{\M^2_{\textup{tot}}}.
\end{equation*}
\end{corollary}

\subsubsection{HSIC-based sensitivity indices}
Another approach for combining kernel embeddings with sensitivity analysis consists in directly using HSIC as a sensitivity index. For example \cite{sdv15} considers the unnormalized index
\begin{equation*}
\Sr_A^{HS} = \HS(\bX_A,Y)
\end{equation*}
relying on a product RKHS $\HH_A=\F_A\times\G$ with kernel $k_{\HH_A}((\bx,y),(\bx',y'))=k_{\mathcal{X}_A}(\bx_A,\bx_A')k_\mathcal{Y}(y,y')$ and provided the following assumption holds:
\begin{assumption}
\label{as:finitekernelxy}
$\forall A\subseteq\Pd$, $\EE_{\xi\sim\textup{P}_{X_A}}k_{\mathcal{X}_A}(\xi,\xi)<\infty$ and $\EE_{\xi\sim\textup{P}_{Y}}k_{\mathcal{Y}}(\xi,\xi)<\infty$.
\end{assumption}
In \cite{sdv15} an empirical normalization inspired by the definition of the distance correlation criterion \citep{dcor07} was also proposed. But similarly to the MMD decomposition above, it is actually possible to exhibit an ANOVA-like decomposition for HSIC, thus providing a natural normalization constant. The main ingredient is an assumption on the kernel $k_\mathcal{X}$ associated to the input variables.

\begin{assumption}
\label{as:HSIC}
The reproducing kernel $k_\mathcal{X}$ of $\F$ is of the form
\begin{equation}
k_\mathcal{X}(\bx,\bx') = \prod_{l=1}^p \left(1 + k_l(x_l,x_l')\right) \label{eq:sumkernel}
\end{equation}
where for each $l=1,\ldots,d$, $k_{l}(\cdot,\cdot)$ is the reproducing kernel of a RKHS $\F_l$ of real functions depending only on variable $x_l$ and such that $1\notin\F_l$.\\
In addition, for all $l=1,\ldots,d$ and $\forall x_l\in\mathcal{X}_l$, we have
\begin{equation}
\int_{\mathcal{X}_l} k_l(x_l,x_l') d\textup{P}_{X_l}(x_l') = 0. \label{eq:nullint}
\end{equation}
\end{assumption}
The first part (\ref{eq:sumkernel}) of Assumption \ref{as:HSIC} may seem stringent, however it can be easily fulfilled by using univariate Gaussian kernels since they define a RKHS which does not contain constant functions \citep{stein06}.

On the contrary, the second assumption (\ref{eq:nullint}) is more subtle. It requires using kernels defining a so-called RKHS of \textit{zero-mean functions} \citep{wa95}. A prominent example of such RKHS is obtained if (a) all input variables are uniformly distributed on $[0,1]$ and (b) the univariate kernels are chosen among the Sobolev kernels with smoothness parameter $r\geq 1$:
\begin{equation}
k_l(x_l,x_l') = \frac{B_{2r}(\vert x_l-x_l'\vert)}{(-1)^{r+1}(2r)!} + \sum_{j=1}^r \frac{B_j(x_l)B_j(x_l')}{(j!)^2} \label{eq:sobkernel}
\end{equation}
where $B_j$ is the Bernoulli polynomial of degree $j$. Even though applying a preliminary transformation on the inputs in order to get uniform variables is conceivable (with \textit{e.g.} the probability integral transform), a more general and elegant procedure has been proposed by \cite{durrande12}. Starting from an arbitrary univariate $k(\cdot,\cdot)$, they build a zero-mean kernel $k_0^D(\cdot,\cdot)$ given by
\begin{equation*}
k_0^D(x,x') = k(x,x') - \frac{\int k(x,t)d\textup{P}(t)\int k(x',t)d\textup{P}(t)}{\iint k(s,t)d\textup{P}(s)d\textup{P}(t)} \label{eq:nullkernel}
\end{equation*}
where $k_0^D(\cdot,\cdot)$ satisfies $\forall x,\ \int k_0^D(x,t)d\textup{P}(t)=0$. Interestingly, they also show that the RKHS $\HH_0$ associated to $k_0(\cdot,\cdot)$ is orthogonal to the constant functions, thus satisfying directly the requirements for the product kernel (\ref{eq:sumkernel}).\\
More recently, several works made use of the Stein operator \citep{stein72} to define the Stein discrepancy in a RKHS \citep{chwia16} which showed great potential for Monte-Carlo integration \citep{oates17} or goodness-of-fit tests \citep{gorham15,chwia16,jit17} when the target distribution is either impossible to sample or is known up to a normalization constant. More precisely, given a RKHS $\HH$ with kernel $k(\cdot,\cdot)$ of functions in $\mathbb{R}^d$ and a (target) probability distribution with density $p(\cdot)$, they define a new RKHS $\HH_0$ with kernel $k_0^S(\cdot,\cdot)$ which writes
\begin{equation*}
k_0^S(\bx,\bx') = \nabla_\bx\nabla_{\bx'}k(\bx,\bx')+ \frac{\nabla_\bx p(\bx)}{p(\bx)} \nabla_{\bx'}k(\bx,\bx') + \frac{\nabla_{\bx'}p(\bx')}{p(\bx')} \nabla_{\bx}k(\bx,\bx')+ \frac{\nabla_\bx p(\bx)}{p(\bx)}\frac{\nabla_{\bx'}p(\bx')}{p(\bx')}  k(\bx,\bx')
\end{equation*}
and it can be proved that $\forall \bx\in\mathbb{R}^d,\ \int k_0^S(\bx,\bx')p(\bx')d\bx'=0$. Unlike $k_0^D(\cdot,\cdot)$, kernel $k_0^S(\cdot,\cdot)$ can still be defined when $p(\cdot)$ is known up to a constant: this property may find interesting applications in GSA when the input distributions are obtained via a preliminary Bayesian data calibration, since it would no longer be required to perform a costly sampling step of their posterior distribution and one could easily use the unnormalized posterior distribution instead.\\

With Assumption \ref{as:HSIC}, we can now state a decomposition for HSIC-based sensitivity indices.
\begin{theorem}[ANOVA decomposition for HSIC]
\label{th:HSICanova}
Under the same assumptions of Theorem \ref{th:anova} (in particular, the random vector $\bX$ has independent components) and with Assumptions \ref{as:finitekernelxy} and \ref{as:HSIC}, the HSIC dependence measure between $\bX=\left(X_1,\ldots,X_d\right)$ and $Y$ can be decomposed as
\begin{equation*}
\HS\left(\bX,Y\right) = \sum_{A\subseteq\Pd} \HS_A
\end{equation*}
where each term is given by
\begin{equation*}
\HS_A = \sum_{B \subset A}(-1)^{|A|-|B|} \HS\left({\bX}_B,Y\right)
\end{equation*}
and $\HS\left({\bX}_B,Y\right)$ is defined with a product RKHS $\HH_B=\F_B\times\G$ with kernel $k_B(\bx_B,\bx_B')k_{\mathcal{Y}}(y,y')=\prod_{l\in B}\left(1+k_l(x_l,x_l')\right)k_{\mathcal{Y}}(y,y')$ as in (\ref{eq:sumkernel}).
\end{theorem}
The proof, which mainly relies on Mercer's theorem and on Theorem 4.1 from \cite{kuo10}, is given in Appendix \ref{sec:HSICanova}. Once again, this decomposition resembles the ANOVA decomposition (\ref{th:anova}), where the conditional variances are replaced with HSIC dependence measures between subsets of inputs and the output.\\
Properly normalized HSIC-based indices can then be defined:
\begin{definition}[HSIC-based sensitivity indices]
\label{defi:normHSICindex}
In the frame of Theorem \ref{th:HSICanova}, let $A\subseteq\Pd$. The normalized HSIC-based sensitivity index associated to a subset $A$ of input variables is defined as
\begin{equation*}
S_A^{\HS} = \frac{\HS_A}{\HS\left(\bX,Y\right)},
\end{equation*}
while the total HSIC-based index associated to $A$ is 
\begin{equation*}
S_A^{\mbox{\scriptsize{T}},\HS} = \sum_{B\subseteq \Pd,\, B\cap A \neq \emptyset}S_B^{\HS} = 1 - \frac{\HS(\bX_{-A},Y)}{\HS\left(\bX,Y\right)}.
\end{equation*}
From Theorem \ref{th:HSICanova}, we have the fundamental identity providing the interpretation of HSIC-based indices as percentage of the explained HSIC dependence measure between $\bX=\left(X_1,\ldots,X_d\right)$ and $Y$:
\begin{equation*}
\sum_{A\subseteq\Pd}S_A^{\HS}=1.
\end{equation*}
\end{definition}

Finally, a noteworthy asymptotic result yields a link between HSIC-based indices and MMD-based ones when the input kernel $k_\mathcal{X}$ degenerates to a dirac kernel, as elaborated in the following proposition.

\begin{proposition}
\label{prop:asympHSIC}
For all subset $A\subseteq\Pd$, let us define a product RKHS $\HH_A=\F_A\times\G$ with kernel $k_A(\bx_A,\bx_A')k_{\mathcal{Y}}(y,y')$. We further assume that $\forall \bx_A\in\mathcal{X}_A$, $p_{\bX_A}(\bx_A)>0$ and that 
\begin{equation}
k_A(\bx_A,\bx_A') = \frac{1}{\sqrt{p_{\bX_A}(\bx_A)}\sqrt{p_{\bX_A}(\bx_A')}}\prod_{l\in A} \frac{1}{h} K\left(\frac{x_l-x_l'}{h}\right) \label{eq:limkernel}
\end{equation}
where $K:\ \R \rightarrow\R$ is a symmetric kernel function satisfying $\int_u K(u)du=1$, and $h>0$. Then we have $\forall A\subseteq\Pd$
\begin{equation*}
\lim_{h\rightarrow 0} \HS(\bX_A,Y)=\EE_{{\bX}_A}\left(\M^2(\textup{P}_{Y},\textup{P}_{Y \vert {\bX}_A})\right)
\end{equation*}
where $\HS(\bX_A,Y)$ is defined with the product RKHS $\HH_A=\F_A\times\G$ and $\M^2(\textup{P}_{Y},\textup{P}_{Y \vert {\bX}_A})$ with the RKHS $\G$.
\end{proposition}
The proof is given in Appendix \ref{sec:asympHSIC}. As a particular case of Proposition \ref{prop:asympHSIC}, one can for example choose a (normalized) Gaussian kernel for $k_\mathcal{X}$ with a standard deviation tending to $0$, or the sinc kernel associated to the RKHS of band-limited continuous functions with a cutoff frequency tending to infinity. Obviously the result also holds if one uses different kernels $K$ for each input $X_l\in\bX_A$ in Eq. (\ref{eq:limkernel}).\\

Although they may seem trivial, Proposition \ref{prop:asympHSIC} and Corollary \ref{cor:MMDsobol} actually justify our claim that both the MMD- and the HSIC-based sensitivity indices are natural generalizations of Sobol' indices, in the sense that a degenerate HSIC-based index with a dirac kernel for the input variables gives the MMD-based index which, in turn, is equal to the Sobol' index when using the dot product kernel for the output.

\subsection{Kernel-embedding Shapley effects} \label{sec:shapley}

In this section, we now discuss how the previously indices can still be valuable in the case where the input variables are no longer independent. In this setting, the Shapley effects introduced in the context of GSA by Owen \citep{owe14} and based on Shapley values \citep{sha53} from game theory have appealing properties, since they provide a proper allocation of the output variance to each input variable, without requiring they are independent. We recall their definition below.

\begin{definition}[Shapley effects \citep{sha53}]
For any $l=1\ldots,d$, the Shapley effect of input $X_l$ is given by
\begin{equation}
Sh_l = \frac{1}{\var\, Y}\frac{1}{p} \sum_{A\subseteq\Pd,\, A\niton l} \binom{p-1}{\vert A\vert}^{-1}\ \bigg\{\var\, \EE\left(Y\vert\bX_{A\cup \{l\}}\right) - \var\,\EE\left(Y\vert\bX_{A}\right)\bigg\}. \label{eq:shapley}
\end{equation}
This definition corresponds to the Shapley value \citep{sha53} 
\begin{equation*}
\phi_l = \frac{1}{p} \sum_{A\subseteq\Pd,\, A\niton l} \binom{p-1}{\vert A\vert}^{-1}\ \bigg\{\textup{val}\left(A\cup \{l\}\right) - \textup{val}\left(A\right)\bigg\}
\end{equation*}
with value function $\textup{val}: \Pd\rightarrow\R_+$ equal to $\textup{val}(A)=\var\,\EE\left(Y\vert\bX_{A}\right)/\var\, Y$. Moreover, we have the following decomposition
\begin{equation*}
\sum_{l=1}^p Sh_l = 1.
\end{equation*}
\end{definition}
The only requirement is that the value function satisfies $\textup{val}: \Pd\rightarrow\R_+$ such that $\textup{val}(\emptyset)=0$. Combining this result with the kernel-based sensitivity indices is consequently straightforward, which leads to the definition of \textit{kernel-embedding Shapley effects}:
\begin{definition}[Kernel-embedding Shapley effects]
\label{def:shap}
For any $l=1\ldots,d$, we define
\begin{itemize}
\item[(a)] The MMD-Shapley effect
\begin{eqnarray}
Sh^{\M}_l &=& \frac{1}{\M^2_{\textup{tot}}}\frac{1}{p} \sum_{A\subseteq\Pd,\, A\niton l} \binom{p-1}{\vert A\vert}^{-1}\ \bigg\{\EE_{{\bX}_{A\cup\{l\}}}\left(\M^2(\textup{P}_{Y},\textup{P}_{Y \vert {\bX}_{A\cup\{l\}}})\right) \nonumber\\
&& - \EE_{{\bX}_A}\left(\M^2(\textup{P}_{Y},\textup{P}_{Y \vert {\bX}_A})\right)\bigg\} \label{eq:MMDshapley}
\end{eqnarray}
provided Assumption \ref{as:finitekernely} holds.
\item[(b)] The HSIC-Shapley effect
\begin{equation}
Sh^{\HS}_l = \frac{1}{\HS\left(\bX,Y\right)}\frac{1}{p} \sum_{A\subseteq\Pd,\, A\niton l} \binom{p-1}{\vert A\vert}^{-1}\ \bigg\{\HS\left(\bX_{A\cup \{l\}},Y\right) - \HS\left(\bX_{A},Y\right)\bigg\} \label{eq:HSICshapley}
\end{equation}
provided Assumptions \ref{as:finitekernelxy} and \ref{as:HSIC} hold.
\end{itemize}
We further have the decompositions
\begin{equation*}
\sum_{l=1}^p Sh^{\M}_l = \sum_{l=1}^p Sh^{\HS}_l = 1.
\end{equation*}
\end{definition}
Just like in the independent setting, kernel-embedding Shapley effects (\ref{eq:MMDshapley}) and (\ref{eq:HSICshapley}) can be seen as general kernelized versions of Shapley effects, since Proposition \ref{prop:asympHSIC} and Corollary \ref{cor:MMDsobol} are still valid when the inputs are dependent.

\begin{remark}
In the machine learning community dedicated to the interpretability of black-box models, an importance measure called Kernel-Shap has been recently introduced \citep{lund17}. Although its naming resembles ours, they designate clearly separated approaches, since the Kernel-Shap measure is a local Shapley effect, and the "Kernel" denomination only refers to an estimation procedure without any links to RKHS.
\end{remark}

\subsection{Enhancing traditional GSA with kernels} \label{sec:specific}

Beyond their theoretical interest in themselves, the kernel-ANOVA decompositions and the associated sensitivity indices also appear powerful from a practical point of view when one carefully examines the potential of using kernels. We give below some insights on how they could enhance traditional GSA studies in several settings. 

\paragraph{Categorical model outputs and target sensitivity analysis.}
In some applications, the model output $Y$ is categorical, meaning that $\mathcal{Y}=\{1,\ldots, K\}$ when the output can take $K$ levels. A simple common instance involves two levels, corresponding to a failure/success situation. Similarly even if $Y$ is not categorical, the objective may be to measure the impact of each input on the fact that the output reaches disjoint regions of interest $\mathcal{R}_1,\ldots,\mathcal{R}_K\subset\mathcal{Y}$, as for example in the case where one focuses on events $\{t_{i+1}>Y>t_i\}$ for thresholds $t_i,\ i=1,\ldots,K$. Such an objective is called \textit{target sensitivity analysis} (TSA, see \cite{mar20}) and can be reformulated in a categorical framework by the change of variable $Z=i$ if $Y\in\mathcal{R}_i$.

The case where $\mathcal{Y}=\{0,1\}$ (or equivalently $Z=\1_{\{Y\in\mathcal{R}\}}$) is frequent in TSA. A straightforward approach is to use Sobol' indices with a $0/1$ output, yielding a first-order Sobol index equal to
\begin{equation}
S_l^{\textup{TSA}} = \frac{\EE_{X_l}\left(\PP(Y=1\vert X_l) - \PP(Y=1)\right)^2}{\PP(Y=1)(1-\PP(Y=1))} \label{eq:sobolTSA}
\end{equation}
see \cite{li12}. But to the best of our knowledge, no systematic procedure is available when the number of levels is greater than two. Without resorting yet to our kernel-based indices, there are at least two roads, which actually lead to the same indices:
\begin{itemize}
\item[(a)] The \textit{one-versus-all} approach, where we compute several Sobol' indices by repeatedly considering $Z=\1_{\{Y=i\}}$ for all $i=1\ldots,K$. We thus have a collection of indices
\begin{equation*}
S_l^{\textup{TSA},[i]} = \frac{\EE_{X_l}\left(\PP(Y=i\vert X_l) - \PP(Y=i)\right)^2}{\PP(Y=i)(1-\PP(Y=i))},
\end{equation*}
and we can aggregate them by normalizing each of them by its own variance, yielding
\begin{equation*}
S_l^{\textup{TSA}} = \frac{\sum_{i=1}^K \PP(Y=i)(1-\PP(Y=i)) S_l^{\textup{TSA},[i]}}{\sum_{i=1}^K \PP(Y=i)(1-\PP(Y=i))} =\frac{\sum_{i=1}^K\EE_{X_l}\left(\PP(Y=i\vert X_l) - \PP(Y=i)\right)^2}{\sum_{i=1}^K \PP(Y=i)(1-\PP(Y=i))}.
\end{equation*}
\item[(b)] The \textit{one-hot encoding} approach, which consists in encoding the categorical output into a multivariate vector of $0/1$ variables and use the aggregated Sobol' indices defined in \cite{gamboa13} on these transformed variables. More precisely if $\mathcal{Y}=\{1,\ldots, K\}$, $Y$ is encoded as a $K$-dimensional vector $\left(Z_1=\1_{Y=1},\ldots,Z_K=\1_{Y=K}\right)$. The aggregated Sobol' indices are then
\begin{equation*}
S_l^{\textup{TSA}} = \frac{\sum_{i=1}^K\var\, \EE(Z_i\vert X_l)}{\sum_{i=1}^K \var\, Z_i} =\frac{\sum_{i=1}^K\EE_{X_l}\left(\PP(Y=i\vert X_l) - \PP(Y=i)\right)^2}{\sum_{i=1}^K \PP(Y=i)(1-\PP(Y=i))},
\end{equation*}
which is exactly the index obtained with the one-versus-all approach.
\end{itemize}
As for the kernel-based indices, the process is less cumbersome, since the only ingredient that requires attention is the choice of a kernel $k_{\mathcal{Y}}(\cdot,\cdot)$ adapted to categorical outputs, which has already been investigated in the kernel literature \citep{song07,song12}. We focus here on the simple \textit{dirac kernel} defined as $k_{\mathcal{Y}}(y,y')=\delta(y,y')$ for categorical values $y,y'\in\{1,\ldots, K\}$, and the corresponding kernel-based indices are then
\begin{itemize}
\item The first-order MMD-based index:
\begin{eqnarray*}
S_l^{\M} &=& \frac{\EE_{X_l}\sum_{i=1}^K\sum_{j=1}^K \delta(i,j) \PP(Y=i\vert X_l)\PP(Y=j\vert X_l) - \sum_{i=1}^K\sum_{j=1}^K \delta(i,j) \PP(Y=i)\PP(Y=j) }{\sum_{i=1}^K \PP(Y=i)- \sum_{i=1}^K\sum_{j=1}^K \delta(i,j) \PP(Y=i)\PP(Y=j) }\\
&=& \frac{\EE_{X_l}\sum_{i=1}^K \PP(Y=i\vert X_l)^2 - \sum_{i=1}^K\PP(Y=i)^2 }{\sum_{i=1}^K  \PP(Y=i) - \sum_{i=1}^K \PP(Y=i)^2 }\\
&=& \frac{\sum_{i=1}^K\EE_{X_l}\left(\PP(Y=i\vert X_l) - \PP(Y=i)\right)^2}{\sum_{i=1}^K \PP(Y=i)(1-\PP(Y=i))},
\end{eqnarray*}
where we retrieve again the one-versus-all Sobol' index.
\item The first-order HSIC-based index:
\begin{eqnarray*}
S_l^{\HS} &=& \int_{\mathcal{X}_l\times\mathcal{X}_l} \sum_{i=1}^K \sum_{j=1}^K k_{\{l\}}(x,x') \delta(i,j) \left[p_{X_l\vert Y=i}(x) - p_{X_l}(x) \right]\\
&& \left[p_{X_l\vert Y=j}(x') - p_{X_l}(x') \right] \PP(Y=i)\PP(Y=j) dxdx'\nonumber\\
&=& \sum_{i=1}^K \PP(Y=i)^2\int_{\mathcal{X}_l\times\mathcal{X}_l} k_{\{l\}}(x,x') \left[p_{X_l\vert Y=i}(x) - p_{X_l}(x) \right] \left[p_{X_l\vert Y=i}(x') - p_{X_l}(x') \right] dxdx'\nonumber\\
&=& \sum_{i=1}^K \PP(Y=i)^2 \M^2\left(\textup{P}_{X_l\vert Y=i},\textup{P}_{X_l}\right),
\end{eqnarray*}
thus extending the result of \cite{spa19} to any number of levels.
\item The MMD- and HSIC- Shapley effects using one of the above indices as building block.
\end{itemize}
Interestingly, it has been shown that Eq. (\ref{eq:sobolTSA}) can also been written, up to a constant, as the Pearson $\chi^2$ divergence between $\textup{P}_{X_l\vert Y=1}$ and $\textup{P}_{X_l}$ \citep{per19,spa20}. This means that $S_l^{\textup{TSA}}=S_l^{\M}$ and $S_l^{\HS}$ essentially have the same interpretation as weighted sums of distances between the initial input distributions and the conditional input distributions (when restricted to an output level), with the Pearson $\chi^2$ divergence and the MMD distance, respectively. But we will see in Section \ref{sec:estim} that the estimation of HSIC-based sensitivity indices is much less prone to the curse of dimensionality and does not require density estimation, as opposed to $S_l^{\textup{TSA}}$ \citep{per19}. Finally, note that another kernel for categorical variables has also been proposed \citep{song07,song12}, but this is actually a normalized dirac kernel which would only modify the weights in the indices above.

\paragraph{Beyond scalar model outputs.}
In many numerical simulation models, some of the outputs are curves representing the temporal evolution of physical quantities of the system such as temperatures, pressures, etc. One can also encounter industrial applications which involve spatial outputs \citep{mar08},. In such cases, the two main approaches in GSA are (a) the ubiquitous point of view, where one sensitivity index is computed for each time step or each spatial location \citep{ter17} and (b) the dimension reduction angle, in which one preliminary projects the output into a low-dimensional vector space and then calculates aggregated sensitivity indices for this new multivariate vector \citep{lam10,gamboa13}.

However, the kernel perspective for such structured outputs can bring new insights for GSA. Indeed the kernel literature has already proposed several ways to handle curves or images in regression or classification tasks. For instance the PCA-kernel \citep{ferraty06} can be used as an equivalent of (b), such as illustrated in \cite{sdv15}. But more interestingly, kernels dedicated to times series were designed, such as the global alignment kernel \citep{cut11} inspired by the dynamic time-warping kernel \citep{sak78}. Such kernel could be employed in industrial applications where one is interested by the impact of an input variable on the shape of the output curve. On the other hand, for dealing with spatial outputs similar to images such as in \cite{mar08}, one may consider a kernel based on image classification \citep{har07} which would be better suited to analyze the impact of inputs on the change of the shapes appearing inside the image output.

Finally, numerical models involving graphs as inputs or outputs (\textit{e.g.} electricity networks or molecules) may now be tractable with GSA by employing kernels specifically tailored for graphs \citep{gar03,ram03}.

\paragraph{Stochastic numerical models.}
On occasions one has to deal with \textit{stochastic simulators}, where internally the numerical model relies on random draws to compute the output. Typical industrial applications include models dedicated to the optimization of maintenance costs, where random failures are simulated during the system life cycle, or molecular modeling to predict macroscopic properties based on statistical mechanics, where several microstates of the system are generated at random \citep{mou15}. For fixed values of the input variables, the output is therefore a probability distribution, meaning that $\mathcal{Y}\subset\mathcal{M}_1^+$ the set of probability measures. In this setting GSA aims at measuring how changes in the inputs modify the output probability distribution, which is clearly out of the traditional scope of GSA. 

Once again the kernel point of view makes it possible to easily recycle the MMD- and the HSIC-based sensitivity indices in this context since they only require the definition of a kernel $k_{\mathcal{Y}}(\cdot,\cdot)$ on probability distributions. This can be achieved through one of the two following kernels:
\begin{equation}
k_{\mathcal{Y}}(\textup{P},\textup{Q}) = \sigma^2 e^{-\lambda \M^2(\textup{P},\textup{Q})} \label{eq:kerneldistr}
\end{equation}
introduced in \cite{song08} or 
\begin{equation*}
k_{\mathcal{Y}}(\textup{P},\textup{Q}) = \sigma^2 e^{-\lambda W_2^2(\textup{P},\textup{Q})}
\end{equation*}
discussed in \cite{bac17} where $\textup{P},\textup{Q}\in\mathcal{M}_1^+$, $W_2$ is the Wasserstein  distance and $\sigma^2,\lambda>0$ are parameters.

\section{Estimation}\label{sec:estim}

The properly normalized kernel-based sensitivity indices being defined above, we now discuss their estimation. The HSIC-based index is first examined as we only consider already proposed estimators. On the other hand, the MMD-based index is analyzed more thoroughly since several estimators can be envisioned given its close links with Sobol' indices. Finally we investigate the estimation of kernel-embedding Shapley effects.

\subsection{HSIC-based index estimation}\label{sec:hsicestim}
We start by observing that if Assumption \ref{as:HSIC} holds, for any subset $A\subseteq\Pd$ we have $\EE_{\bX_A}k_A(\bX_A,\bx_A')=1$ for all $\bx'_A\in\mathcal{X}_A$, which means that HSIC in Eq. (\ref{eq:hsic}) simplifies into:
\begin{eqnarray*}
\HS(\bX_A,Y) = \EE_{\bX_A,\bX_A',Y,Y'} k_A(\bX_A,\bX_A')k_{\mathcal{Y}}(Y,Y') - \EE_{Y,Y'}k_{\mathcal{Y}}(Y,Y').
\end{eqnarray*}
Given a sample $\left(\bx^{(i)},y^{(i)}\right)$, $i=1,\ldots,n$ and following \cite{song07,gretton08} two estimators $\HS_u(\bX_A,Y)$ and $\HS_b(\bX_A,Y)$ based on U- and V-statistics, respectively, can be introduced:
\begin{eqnarray*}
\HS_u(\bX_A,Y) &=& \frac{1}{n(n-1)}\sum_{i,j=1,\, i\neq j}^n \left(k_A(\bx_A^{(i)},\bx_A^{(j)})-1\right) k_{\mathcal{Y}}(y^{(i)},y^{(j)})\\
\HS_b(\bX_A,Y) &=& \frac{1}{n^2}\sum_{i,j=1}^n \left(k_A(\bx_A^{(i)},\bx_A^{(j)})-1\right) k_{\mathcal{Y}}(y^{(i)},y^{(j)})
\end{eqnarray*}
where we assume that $\textup{P}_{\bX_A}$ is known and is used to compute analytically the zero-mean kernels in Eq. (\ref{eq:nullkernel}). The study of the version of the above estimators when the sample $\left(\bx^{(i)}\right)_{i=1,\ldots,n}$ also serves to estimate $k_A$ is left as future work.
Both $\HS_u(\bX_A,Y)$ and $\HS_b(\bX_A,Y)$ converge in probability to $\HS(\bX_A,Y)$ with rate $1/\sqrt{n}$, and one can show \citep{song07} that if we assume that $k_A$ and $k_{\mathcal{Y}}$ are bounded almost everywhere by $1$ and are nonnegative, with probability at least $1-\delta$ we have
\begin{equation*}
\vert \HS_u(\bX_A,Y) - \HS(\bX_A,Y)\vert \leq 8 \sqrt{\log(2/\delta)/n}.
\end{equation*}
The asymptotic distributions of $\HS_u(\bX_A,Y)$ and $\HS_b(\bX_A,Y)$ have also been studied in the case where $\bX_A$ and $Y$ are dependent, see \cite{song07} and \cite{gretton08}.

It is worth mentioning that here the number of model evaluations is $n$, which is independent from the input dimension, meaning that all HSIC-based sensitivity indices can be computed with only a given sample $\left(\bx^{(i)},y^{(i)}\right)$, $i=1,\ldots,n$.

\subsection{MMD-based index estimation}\label{sec:mmdestim}

MMD-based indices are close generalizations of Sobol' indices, since they involve computing the expectation of a conditional quantity (a MMD distance with a conditional probability for the former and a conditional variance for the latter). This is the reason why estimation procedures developed for Sobol' indices can be adapted to the MMD ones. The first two estimators discussed below assume that one can easily sample the computer model for any input values (to be determined by the estimation procedure), as opposed to the next two ones which can be defined with any given sample $\left(\bx^{(i)},y^{(i)}\right)$, $i=1,\ldots,n$.

\subsubsection{Double-loop Monte-Carlo}

The first naive estimator consists in systematically resampling the conditional distribution $\textup{P}_{Y \vert {\bX}_A=\bx_A}$ for many values of $\bx_A$, as detailed in Algorithm \ref{alg:doubleloop} below.

\begin{algorithm}
\caption{Double-loop estimator of $\EE_{\bX_A}\left(\textup{MMD}^2(\textup{P}_Y,\textup{P}_{Y \vert \bX_A})\right)$}
\label{alg:doubleloop}
\begin{algorithmic}
\STATE Sample $\bx^{(j)}$ from $\textup{P}_{\bX}$ and compute $y^{(j)}=\eta(\bx^{(j)})$ for $j=1,\ldots,m$.
\FOR{$i=1\ldots,n$}
\STATE \textit{Outer-loop}
\STATE Sample $\bx_A^{(i)}$ from $\textup{P}_{\bX_A}$;
\FOR{$j=1\ldots,m$}
\STATE \textit{Inner-loop}
\STATE Sample $\bx_{-A}^{(j)}$ from $\textup{P}_{\bX_{-A}}$ (if inputs are independent) or from $\textup{P}_{\bX_{-A}\vert\bX_A=\bx_A^{(i)}}$ (otherwise);
\STATE Compute $\tilde{y}^{(j)}=\eta(\bx')$ where $\bx'_A=\bx_A^{(i)}$ and $\bx'_{-A}=\bx_{-A}^{(j)}$;
\ENDFOR
\STATE Compute 
\begin{equation*}
M^{(i)}=\frac{1}{n^2} \sum_{j,j'=1}^m k_{\mathcal{Y}}\left(y^{(j)},y^{(j')}\right)+\frac{1}{n^2} \sum_{j,j'=1}^m k_{\mathcal{Y}}\left(\tilde{y}^{(j)},\tilde{y}^{(j')}\right) - \frac{2}{n^2} \sum_{j,j'=1}^m k_{\mathcal{Y}}\left(y^{(j)},\tilde{y}^{(j')}\right)
\end{equation*}
the estimator of $\textup{MMD}^2(\textup{P}_Y,\textup{P}_{Y \vert \bX_A=\bx_A^{(i)}})$;
\ENDFOR
\STATE $\EE_{\bX_A}\left(\textup{MMD}^2(\textup{P}_Y,\textup{P}_{Y \vert \bX_A})\right)$ is finally estimated by $\frac{1}{n} \sum_{i=1}^n M^{(i)}$.
\end{algorithmic}
\end{algorithm}
For each MMD-based index of a subset of variables $\bX_A$ the total number of model evaluations is $(n+1)m$, which means for example that all first-order MMD-based sensitivity indices are computed at a cost of $p(n+1)m$ model evaluations. It is however possible to design better sampling strategies to compute first-order and total indices if the inputs are independent, as explained in the next section.

\subsubsection{Pick-freeze estimators}\label{sec:pickfreeze}

We begin by recalling the definition of the pick-freeze estimators for Sobol' indices.
\begin{lemma}[Pick-freeze formulation of Sobol indices \citep{janon14}]
\label{lemma:pfsobol}
Assume $\bX$ and $\bX'$ are two independent copies of the input vector, the inputs being independent. For any subset $A\subseteq\Pd$ define $\bX^{\sim A}$ the vector assembled from $\bX$ and $\bX'$ such that $\bX^{\sim A}_A=\bX_A$ and $\bX^{ \sim A}_{-A}=\bX'_A$. Now if we denote $Y=\eta(\bX)$ and $Y^{\sim A}=\eta(\bX^{\sim A})$, we have
\begin{eqnarray*}
\var\, \EE\left(Y\vert \bX_A\right)&=& \textup{Cov}\left(Y,Y^{\sim A}\right),\\
S^T_A &=& 1 - \frac{\textup{Cov}\left(Y,Y^{\sim -A}\right)}{\var\, Y}.
\end{eqnarray*}
\end{lemma}
In the particular case of $A=\{l\}$, the first-order and total indices $S_l$ and $S^T_l$ can be estimated by collecting estimators $\hat{V}_l$, $\hat{V}_{-l}$ and $\hat{V}$ of  $\textup{Cov}\left(Y,Y^{\sim l}\right)$, $\textup{Cov}\left(Y,Y^{\sim -l}\right)$ and $\var\, Y$, respectively. Such estimators have been first studied in \cite{homma96}, but we focus on the ones introduced by \cite{sal10} which write
\begin{eqnarray*}
\hat{V}_l &=& \frac{1}{n} \sum_{i=1}^n \eta(\bx^{(i)}) \left\{\eta(\bx^{\sim l,(i)}) - \eta(\bx'^{(i)})\right\},\\
\hat{V}_{-l} &=& \frac{1}{n} \sum_{i=1}^n \eta(\bx'^{(i)}) \left\{\eta(\bx^{\sim l,(i)}) - \eta(\bx^{(i)})\right\},\\
\hat{V} &=& \frac{1}{n} \sum_{i=1}^n \eta(\bx^{(i)})^2 -  \left(\frac{1}{n} \sum_{i=1}^n \eta(\bx^{(i)})\right)^2
\end{eqnarray*}
where $\bx^{(i)}$ and $\bx'^{(i)}$ denote independent samples of $\bX$ and $\bx^{\sim l,(i)}$ is a vector such that $\bx^{\sim l,(i)}_l=\bx^{(i)}_l$ and $\bx^{ \sim l,(i)}_{-l}=\bx'^{(i)}_{-l}$. The total number of model evaluations to estimate both $S_l$ and $S^T_l$ is thus $(p+2)n$, which is much less than the amount required by the previously introduced double-loop estimator.

We now build upon these estimators to design equivalent ones for the first-order and total MMD-based sensitivity indices. The main ingredient is to state an equivalent of Lemma \ref{lemma:pfsobol} for the MMD.
\begin{lemma}[Pick-freeze formulation of MMD-based indices]
\label{lemma:pfmmd}
With the same notations and assumptions as in Lemma \ref{lemma:pfsobol}, we have
\begin{equation*}
\EE_{\bX_A}\left(\textup{MMD}^2(\textup{P}_Y,\textup{P}_{Y \vert \bX_A})\right) = \EE k_{\mathcal{Y}}\left(Y,Y^{\sim A}\right)-\EE k_{\mathcal{Y}}\left(Y,Y'\right).
\end{equation*}
\end{lemma}
\begin{proof}
Since $Y$ and $Y^{\sim A}$ are conditionally independent on $\bX_A$ with the same distribution, we can write $\EE k_{\mathcal{Y}}\left(Y,Y^{\sim A}\right)=\EE_{\bX_A}\EE\left[k_{\mathcal{Y}}\left(Y,Y^{\sim A}\right) \vert \bX_A\right]=\EE_{\bX_A} \EE_{\zeta,\zeta'\sim \textup{P}_{Y\vert \bX_A}}k_{\mathcal{Y}}(\zeta,\zeta') $.
\end{proof}
Estimators $\widehat{\M}^2_l$ for $\EE_{X_l}\left(\textup{MMD}^2(\textup{P}_Y,\textup{P}_{Y \vert X_l})\right)$ and $\widehat{\M}^2_{-l}$ for $\EE_{X_{-l}}\left(\textup{MMD}^2(\textup{P}_Y,\textup{P}_{Y \vert X_{-l}})\right)$ are therefore given by
\begin{eqnarray*}
\widehat{\M}^2_l &=& \frac{1}{n} \sum_{i=1}^n \left\{k_{\mathcal{Y}}\left(\eta(\bx^{(i)}),\eta(\bx^{\sim l,(i)})\right)-k_{\mathcal{Y}}\left(\eta(\bx^{(i)}),\eta(\bx'^{(i)})\right)\right\} \\
\widehat{\M}^2_{-l} &=& \frac{1}{n} \sum_{i=1}^n \left\{k_{\mathcal{Y}}\left(\eta(\bx'^{(i)}),\eta(\bx^{\sim l,(i)})\right)-k_{\mathcal{Y}}\left(\eta(\bx^{(i)}),\eta(\bx'^{(i)})\right)\right\} 
\end{eqnarray*}
Similarly the normalization constant $\M^2_{\textup{tot}}=\EE k_{\mathcal{Y}}(Y,Y)-\EE k_{\mathcal{Y}}(Y,Y')$ is estimated by
\begin{equation*}
\widehat{\M}^2_{\textup{tot}} = \frac{1}{n} \sum_{i=1}^n k_{\mathcal{Y}}\left(\eta(\bx^{(i)}),\eta(\bx^{(i)})\right) - \frac{1}{n^2} \sum_{i,j=1}^n k_{\mathcal{Y}}\left(\eta(\bx^{(i)}),\eta(\bx^{(j)})\right).
\end{equation*}
All these estimators can actually be recovered by using Mercer's theorem $k_{\mathcal{Y}}(y,y')=\sum_{r=1}^\infty \phi_r(y)\phi_r(y')$ and plugging the Sobol' estimators of $\textup{Cov}\left(\phi_r(Y),\phi_r(Y^{\sim l})\right)$, $\textup{Cov}\left(\phi_r(Y),\phi_r(Y^{\sim -l})\right)$ and $\var\, \phi_r(Y)$ for all $r>1$. Once again all first-order and total MMD-based sensitivity indices can be estimated with a total cost of $(p+2)n$ model evaluations, and by the strong law of large numbers it is straightforward to show that both $\widehat{\M}^2_l/\widehat{\M}^2_{\textup{tot}} $ and $\widehat{\M}^2_{-l}/\widehat{\M}^2_{\textup{tot}} $ are consistent.

\subsubsection{First-order index estimation with ranks}\label{sec:mmdrank}

The two previous estimators, although simple, necessitate specific sampling schemes (double-loop Monte-Carlo or pick-freeze) which may not be amenable in practice. In addition first-order MMD indices estimation call for a number of model evaluations which increases with the number of input variables $d$. Recently, \cite{gamboa20} introduced new estimators of first-order Sobol' indices based on ranking and inspired by the work of \cite{chat20}. In particular, for any pair of random variables $(V,Y)$ and measurable bounded functions $f$ and $g$, they propose a universal estimation procedure for expectations of the form 
\begin{equation*}
\EE\left(\EE[f(Y)\vert V] \EE[g(Y)\vert V]\right)
\end{equation*}
using only a given sample $(v^{(i)},y^{(i)})_{i=1,\ldots,n}$ and an estimator given by
\begin{equation*}
\frac{1}{n} \sum_{i=1}^n f(y^{(i)})g(y^{(\sigma_n(i))})
\end{equation*} 
where $\sigma_n$ is a random permutation with no fixed point and measurable with respect to the $\sigma$-algebra generated by $(v^{(1)},\ldots,v^{(n)})$. First-order Sobol' indices are then estimated using $f(x)=g(x)=x$ and the permutation $\sigma_n=N$ defined as in \cite{chat20}:
\begin{equation}
\label{eq:perm}
N(i) = \left\{
    \begin{array}{ll}
        \pi^{-1}(\pi(i)+1)  & \mbox{if } \pi(i)+1\leq n \\
        \pi^{-1}(1) & \mbox{otherwise}
    \end{array}
\right.
\end{equation}
where $\pi(i)$ is the rank of $V^{(i)}$ in the sample $(V^{(1)},\ldots,V^{(n)})$. All first-order indices are finally obtained with a given sample by considering one after the other the pairs $(X_l,Y)$ with their own permutation based on the sample ranks of $X_l$.

Interestingly, it is possible to generalize this result to the first-order MMD indices with the following proposition.
\begin{proposition}[Generalization of Proposition 3.2 from \cite{gamboa20}]
\label{prop:rank}
Let $k(\cdot,\cdot)$ be a measurable bounded kernel and $(v^{(i)},y^{(i)})_{i=1,\ldots,n}$ an iid sample from a pair of random variables $(V,Y)$. Consider a random permutation with no fixed point and measurable with respect to the $\sigma$-algebra generated by $(v^{(1)},\ldots,v^{(n)})$ such that for any $i=1,\ldots,n$, $v^{(\sigma_n(i))} \rightarrow v^{(i)}$ as $n\rightarrow\infty$ with probability one. Then the estimator
\begin{equation*}
\chi_n(V,Y,k) = \frac{1}{n}\sum_{i=1}^n k(y^{(i)},y^{(\sigma_n(i))})
\end{equation*}
converges almost surely to 
\begin{equation*}
\chi(V,Y,k) = \EE_{V} \EE_{\xi,\xi'\sim \textup{P}_{Y\vert V}}k_{\mathcal{Y}}(\xi,\xi') 
\end{equation*}
as $n\rightarrow\infty$.
\end{proposition}
The proof relies again on Mercer's theorem and is given in Appendix \ref{sec:rank}. The estimators of $\EE_{X_l}\left(\textup{MMD}^2(\textup{P}_Y,\textup{P}_{Y \vert X_l})\right)$ and $\M^2_{\textup{tot}}$ are finally given by 
\begin{eqnarray}
\widehat{\M}^2_l &=& \frac{1}{n} \sum_{i=1}^n k_{\mathcal{Y}}\left(y^{(i)},y^{(\sigma^l_n(i))}\right)- \frac{1}{n^2} \sum_{i,j=1}^n k_{\mathcal{Y}}\left(y^{(i)},y^{(j)}\right)\label{eq:mmdrank}\\
\widehat{\M}^2_{\textup{tot}} &=& \frac{1}{n} \sum_{i=1}^n k_{\mathcal{Y}}\left(y^{(i)},y^{(i)}\right) - \frac{1}{n^2} \sum_{i,j=1}^n k_{\mathcal{Y}}\left(y^{(i)},y^{(j)}\right)\label{eq:mmdtotestim}
\end{eqnarray}
for a sample $\left(\bx^{(i)},y^{(i)}\right)$, $i=1,\ldots,n$ and where $\sigma^l_n$ is the permutation defined in Eq. (\ref{eq:perm}) with a ranking performed on the sample $\left(x_l^{(i)}\right)_{i=1,\ldots,n}$.

\subsubsection{Higher-order index estimation with nearest-neighbors}

The ranking approach introduced above can actually be generalized to estimate higher-order sensitivity indices by replacing ranking (in dimension 1) by nearest-neighbors (in arbitrary dimension), since they define a permutation with the same properties as required in Proposition \ref{prop:rank}. This was proposed independently by \cite{aza19} in the context of a dependence measure and by \cite{bro20} for Shapley effects estimation. Here we adopt the formalism of \cite{bro20}, where they introduce $j^*_A(i,m)$ the index such that the sample point $\bx_A^{(j^*_A(i,m))}$ of the subset $A\subseteq\Pd$ of input variables is the $m$-th nearest neighbor of the sample point $\bx_A^{(i)}$ in a sample of the inputs $\left(\bx^{(i)}\right)_{i=1,\ldots,n}$. Then their nearest-neighbor estimator $\hat{V}^{\textup{knn}}_{A}$ of $\var\, \EE(Y\vert\bX_A)$ is given by 
\begin{equation*}
\hat{V}^{\textup{knn}}_{A} = \frac{1}{n_A} \sum_{j=1}^{n_A} \eta\left(\bx^{(j^*_A(s(j),1))}\right)\eta\left(\bx^{(j^*_A(s(j),2))}\right) - \left(\frac{1}{n} \sum_{i=1}^n \eta\left(\bx^{(i)}\right)\right)^2
\end{equation*}
where $s(j)$, $j=1,\ldots,n_A$ is a sample of uniformly distributed integers in $\{1,\ldots,n\}$, with $n_A\leq n$. The choice of using a subsample $s(j)$ is motivated by the authors so that their framework is general enough for the different aggregation procedures they propose for Shapley effects and for their consistency proofs. Several numerical experimentations not reported here also show that using all the samples instead of subsamples yield biased estimators, so we follow the procedure of \cite{bro20}. Once again this estimator can be generalized to MMD-based indices, where $\EE_{\bX_A}\left(\textup{MMD}^2(\textup{P}_Y,\textup{P}_{Y \vert \bX_A})\right)$ is estimated by
\begin{equation*}
\widehat{\M}^2_A = \frac{1}{n_A} \sum_{j=1}^{n_A} k_{\mathcal{Y}}\left(y^{(j^*_A(s(j),1))},y^{(j^*_A(s(j),2))}\right) - \frac{1}{n^2} \sum_{i,j=1}^n k_{\mathcal{Y}}\left(y^{(i)},y^{(j)}\right)
\end{equation*}
where we denote $y^{(i)}=\eta\left(\bx^{(i)}\right)$. The consistency of this estimator directly follows from the consistency of $\hat{V}^{\textup{knn}}_{A}$ from \cite{bro20} and Mercer's theorem. Since $j^*_A(s(j),1)=s(j)$, the estimator is identical to the ranking-based one in (\ref{eq:mmdrank}) where the permutation from rankings is simply replaced by the index of the nearest neightbor not including itself $j^*_A(s(j),2)$.

\subsection{Shapley effect estimation} \label{sec:shapleyestim}

The last estimation task concerns kernel-embedding Shapley effects set forth in Definition \ref{def:shap}. Of course a straightforward approach consists in using any of the estimators discussed before in the general formulation of the MMD- or HSIC-Shapley effects. But a closer inspection actually reveals that although this is easy for the HSIC-Shapley effects since both $\HS_u(\bX_A,Y)$ and $\HS_b(\bX_A,Y)$ can be computed for all subsets $A\subseteq\Pd$ with only one sample $\left(\bx^{(i)},y^{(i)}\right)$, $i=1,\ldots,n$, the MMD-Shapley effects require estimators of $\EE_{\bX_A}\left(\textup{MMD}^2(\textup{P}_Y,\textup{P}_{Y \vert \bX_A})\right)$ which do not involve two many calls to the numerical model. Among the estimators introduced in Section \ref{sec:mmdestim}, only the one based on nearest neighbors has a computational cost independent of the number of input variables. This is exactly the framework proposed in \cite{bro20} for the variance-based Shapley effects.

However, as pointed out in \cite{song16} in the case of variance-based Shapley effects, a double-loop Monte-Carlo estimator of the value function $\textup{val}(A)=\var\,\EE\left(Y\vert\bX_{A}\right)/\var\, Y$ can be heavily biased. They show that another value function $\textup{val}'(A)=\EE\var\,\left(Y\vert\bX_{-A}\right)/\var\, Y$ behaves better and gives rise to the exact same Shapley effects (Theorem 1 in \cite{song16}). This is why \cite{bro20} also introduced a nearest neighbor estimator of $\EE\var\,\left(Y\vert\bX_{-A}\right)$ given by
\begin{eqnarray*}
\hat{E}^{\textup{knn}}_{A} =\frac{1}{n_A} \sum_{j=1}^{n_A} \left\{\frac{1}{n_I-1} \sum_{i=1}^n \left[y^{(j^*_{-A}(s(j),i))}-\frac{1}{n_I} \sum_{i=1}^n y^{(j^*_{-A}(s(j),i))}\right]^2\right\}
\end{eqnarray*}
where this time $n_I$ nearest neighbors are used. In a nutshell, the nearest neighbors are used as if they were independent samples from $P_{Y\vert \bX_A=\bx^{(s(j))}}$, which explains why we compute their empirical variance in the formula above. In order to follow the same road for the estimation of MMD-Shapley effects, we first need an equivalent of Theorem 1 from \cite{song16} for a new value function related to the MMD. 
\begin{lemma}[Other formulation of MMD-Shapley effects]
\label{lemma:mmshap}
The Shapley values obtained with value function $\textup{val}'(A)= \EE_{\bX_{-A}}\left[\EE_{\xi\sim\textup{P}_{Y \vert {\bX}_{-A}}} k_{\mathcal{Y}}(\xi,\xi) - \EE_{\xi,\xi'\sim\textup{P}_{Y \vert {\bX}_{-A}}} k_{\mathcal{Y}}(\xi,\xi')  \right]/\M^2_{\textup{tot}}$ are exactly equal to the MMD-Shapley effects from Definition \ref{def:shap} with value function $\textup{val}(A)=\EE_{{\bX}_A}\left(\M^2(\textup{P}_{Y},\textup{P}_{Y \vert {\bX}_A})\right)/\M^2_{\textup{tot}}$.
\end{lemma}
The proof is based on the generalization of the law of total variance for the generalized variance $\M^2_{\textup{tot}}$ and is given in Appendix \ref{sec:mmshap}. A nearest neighbor estimator $\widehat{\textup{E}\M}^2_A$ of 
\begin{equation*}
\EE_{\bX_{-A}}\left[\EE_{\xi\sim\textup{P}_{Y \vert {\bX}_{-A}}} k_{\mathcal{Y}}(\xi,\xi) - \EE_{\xi,\xi'\sim\textup{P}_{Y \vert {\bX}_{-A}}} k_{\mathcal{Y}}(\xi,\xi')  \right]
\end{equation*}
is then given by
\begin{eqnarray*}
\widehat{\textup{E}\M}^2_A &=& \frac{1}{n_A} \sum_{j=1}^{n_A} \left\{\frac{1}{n_I} \sum_{i=1}^{n_I} k_{\mathcal{Y}}\left(y^{(j^*_{-A}(s(j),i))},y^{(j^*_{-A}(s(j),i))}\right)\right. \\
&& \left. - \frac{1}{n_I^2} \sum_{i,i'=1}^{n_I} k_{\mathcal{Y}}\left(y^{(j^*_{-A}(s(j),i))},y^{(j^*_{-A}(s(j),i'))}\right) \right\}
\end{eqnarray*}
and the MMD-Shapley effect estimator is
\begin{equation*}
\widehat{Sh}^{\M}_l = \frac{1}{\widehat{\M}^2_{\textup{tot}}}\frac{1}{p} \sum_{A\subseteq\Pd,\, A\niton l} \binom{p-1}{\vert A\vert}^{-1}\ \bigg\{\widehat{\textup{E}\M}^2_{A\cup\{l\}}- \widehat{\textup{E}\M}^2_A\bigg\}.
\end{equation*}
where $\widehat{\M}^2_{\textup{tot}}$ is estimated as in Eq. (\ref{eq:mmdtotestim}).

As a side-note, when the number of input variables is large, the number of terms involved in Shapley effects severely increases and the computational cost to assemble all the terms (even if one uses estimators relying on a given sample only) becomes prohibitive. For such cases it is possible to use a formulation of Shapley effects involving a sum on permutations of $\{1,\ldots,d\}$ instead of a sum on subsets of $\Pd$, which makes it possible to add another level of approximation by computing the sum on a random sample of permutations instead of on all of them \citep{castro09}. Obviously since this trick does not depend on the value function used inside the Shapley values, it can also be used for our kernel-embedding Shapley effects.

\section{Experiments} \label{sec:exp}

In this section we illustrate the behavior of the kernel-based sensitivity indices on several test cases representative of typical GSA industrial applications. In particular, we address the following numerical model categories: a standard scalar output model, a stochastic simulator, a model with a time-series output and a multi-class categorical output simulator with dependent inputs. All the results presented here are reproducible with the \textsf{R} code provided in the supplementary material.

\subsection{Standard scalar output model}
To exemplify the additional insight provided by these indices we first consider a classical GSA test case, the Ishigami function \citep{ishi90} where the output $Y$ is given by
\begin{equation*}
Y = \sin(X_1) + 7\sin(X_2)^2 + X_3^4\sin(X_1)
\end{equation*} 
where $X_l\sim\mathcal{U}(-\pi,\pi)$ for $l=1,\ldots,4$, meaning that we add a dummy input variable $X_4$ for analysis purposes.\\

We start by computing the traditional Sobol' first-order and total sensitivity indices using a pick-freeze estimator as in Section \ref{sec:pickfreeze} with a sample size $n=1000$ and we repeat this calculation 50 times. For each replication the total number of calls to the numerical model is thus $(p+2)n=6000$. We then use the same pick-freeze procedure to estimate the MMD-based first-order and total indices with the exact same samples. For the output we use a Gaussian kernel $k_{\mathcal{Y}}(y,y') = \exp(-\frac{1}{2\sigma^2}(y-y')^2)$ where $\sigma$ is chosen as the median of the pairwise distances between the output samples. Results are given in Figure \ref{fig:example1_MMD}. First note that, as is well known, the first-order Sobol' index of $X_3$ is zero, while its total index is around $0.25$ due to its interaction with $X_1$. $X_2$ is also an important variable, which does not have any interaction since its total Sobol' index is equal to its first-order one. As expected $X_4$ is correctly detected as non-important. The MMD-based indices however bring a different insight: from a probability distribution perspective, one can observe that interactions are much more present since there is a large gap between total and first-order indices for all inputs (except $X_4$ of course). In addition, this time $X_3$ is detected to have a main effect: indeed even though it does not impact the output conditional mean, it influences the tails of the output conditional distribution when it is close to $-2\pi/2\pi$ as was already illustrated in \cite{sdv16}. This shows that MMD-based indices capture other types of input influence than Sobol' ones.\\

\begin{figure}[h!]
\centering
\begin{subfigure}[b]{0.49\textwidth}
         \centering
         \includegraphics[width=\textwidth]{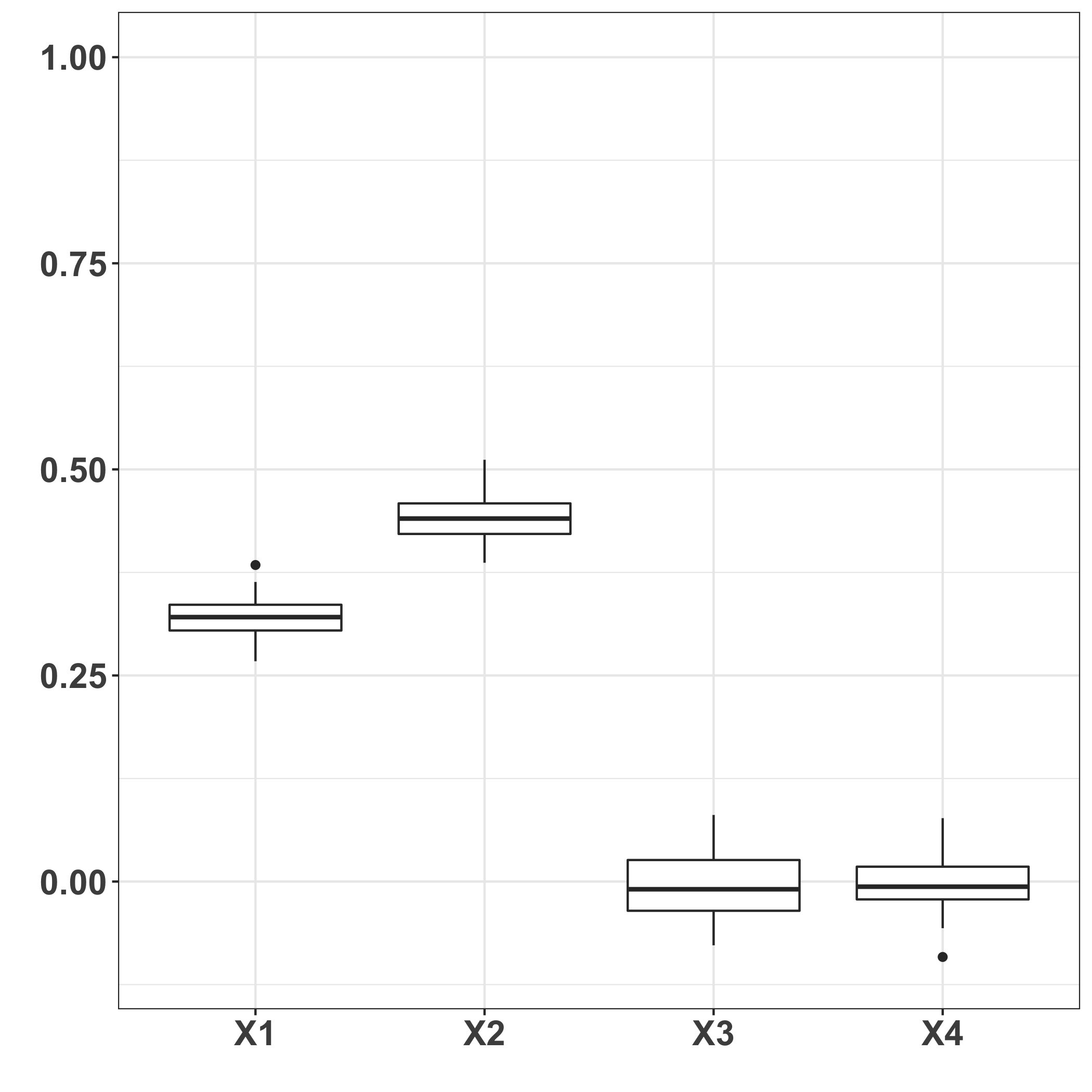}
         \caption{Sobol' first-order index}
     \end{subfigure}
     \hfill
     \begin{subfigure}[b]{0.49\textwidth}
         \centering
         \includegraphics[width=\textwidth]{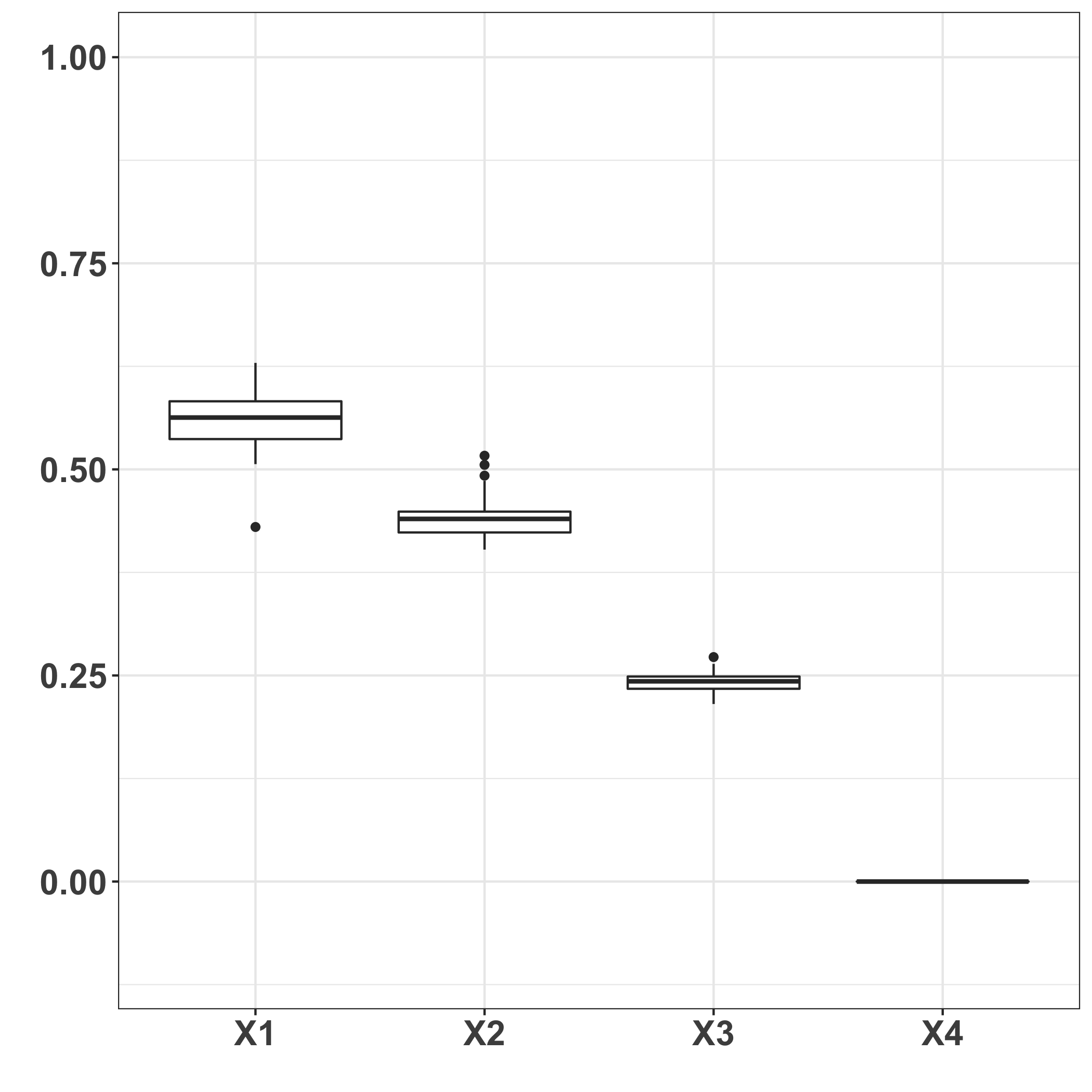}
         \caption{Sobol' total index}
     \end{subfigure}
\begin{subfigure}[b]{0.49\textwidth}
         \centering
         \includegraphics[width=\textwidth]{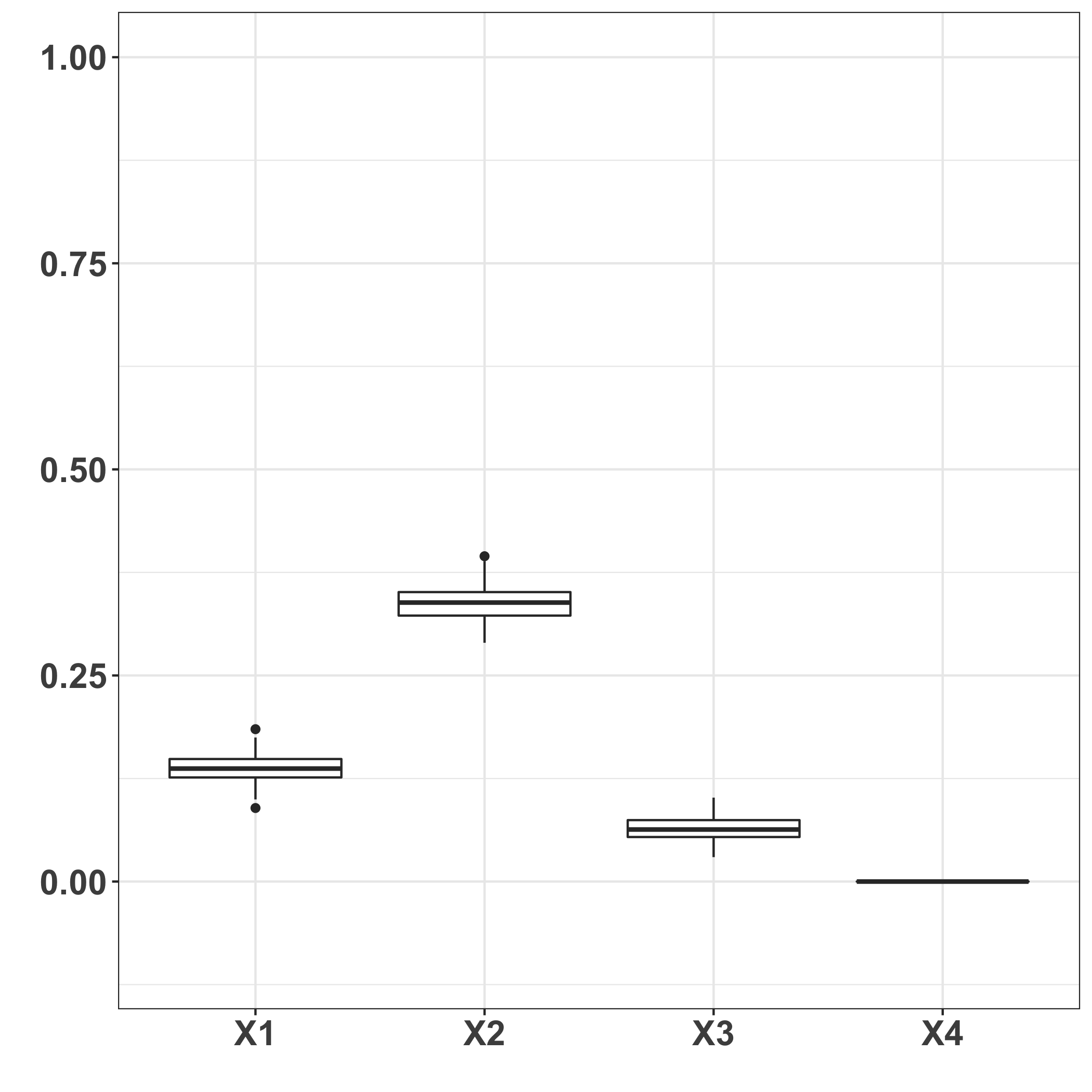}
         \caption{MMD-based first-order index}
     \end{subfigure}
     \hfill
     \begin{subfigure}[b]{0.49\textwidth}
         \centering
         \includegraphics[width=\textwidth]{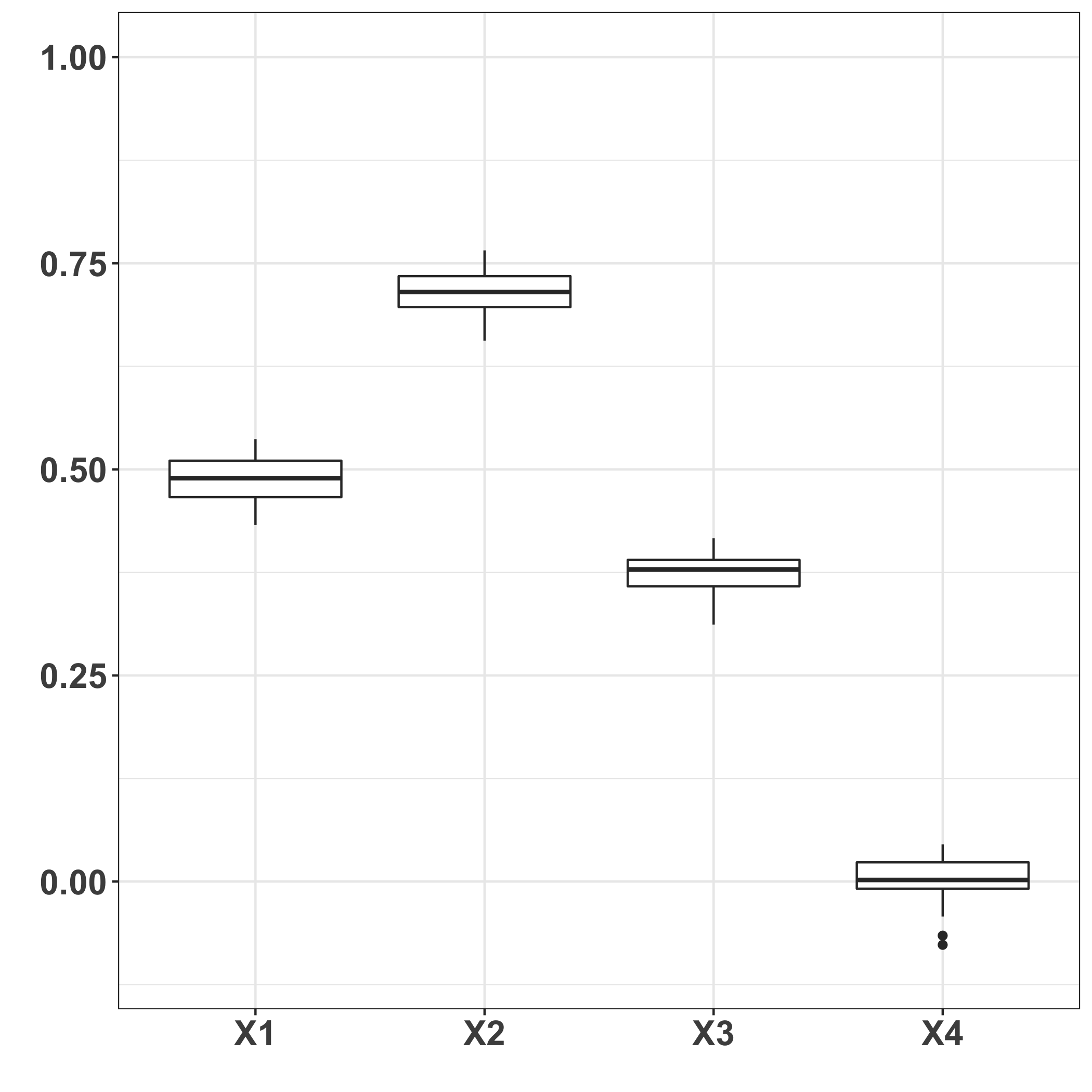}
         \caption{MMD-based total index}
     \end{subfigure}
\caption{Ishigami test case. First-order (a) and total (b) Sobol' indices and first-order (c) and total (d) MMD-based indices with pick-freeze estimators, $n=1000$, 50 replicates.} \label{fig:example1_MMD}
\end{figure}

To take a different view at the inputs/output relationship we also estimate HSIC-based first-order and total indices using the V-statistic of Section \ref{sec:hsicestim}. Again for the output we use the same Gaussian kernel as above, while we use the Sobolev kernel from Eq. (\ref{eq:sobkernel}) for the inputs. Since they are uniform it is easy to renormalize them to satisfy the zero-mean kernel condition. We use only one sample of size $n=1000$ and estimates obtained with 50 replications are reported in Figure \ref{fig:example1_HSIC}. Interestingly, we observe first that with HSIC we no longer detect any interaction: our intuition is that first-order HSIC indices already aggregate a very large family of potential influences and thus interactions may only appear with highly complicated inputs/output link functions. This is supported by the fact that HSIC indices rank the inputs the exact same way at total Sobol' indices. Another appealing property is that to compute all HSIC indices we only need a given sample of moderate size, which is interesting from a screening perspective for GSA on very time-consuming numerical models.

\begin{figure}[h!]
\centering
\begin{subfigure}[b]{0.49\textwidth}
         \centering
         \includegraphics[width=\textwidth]{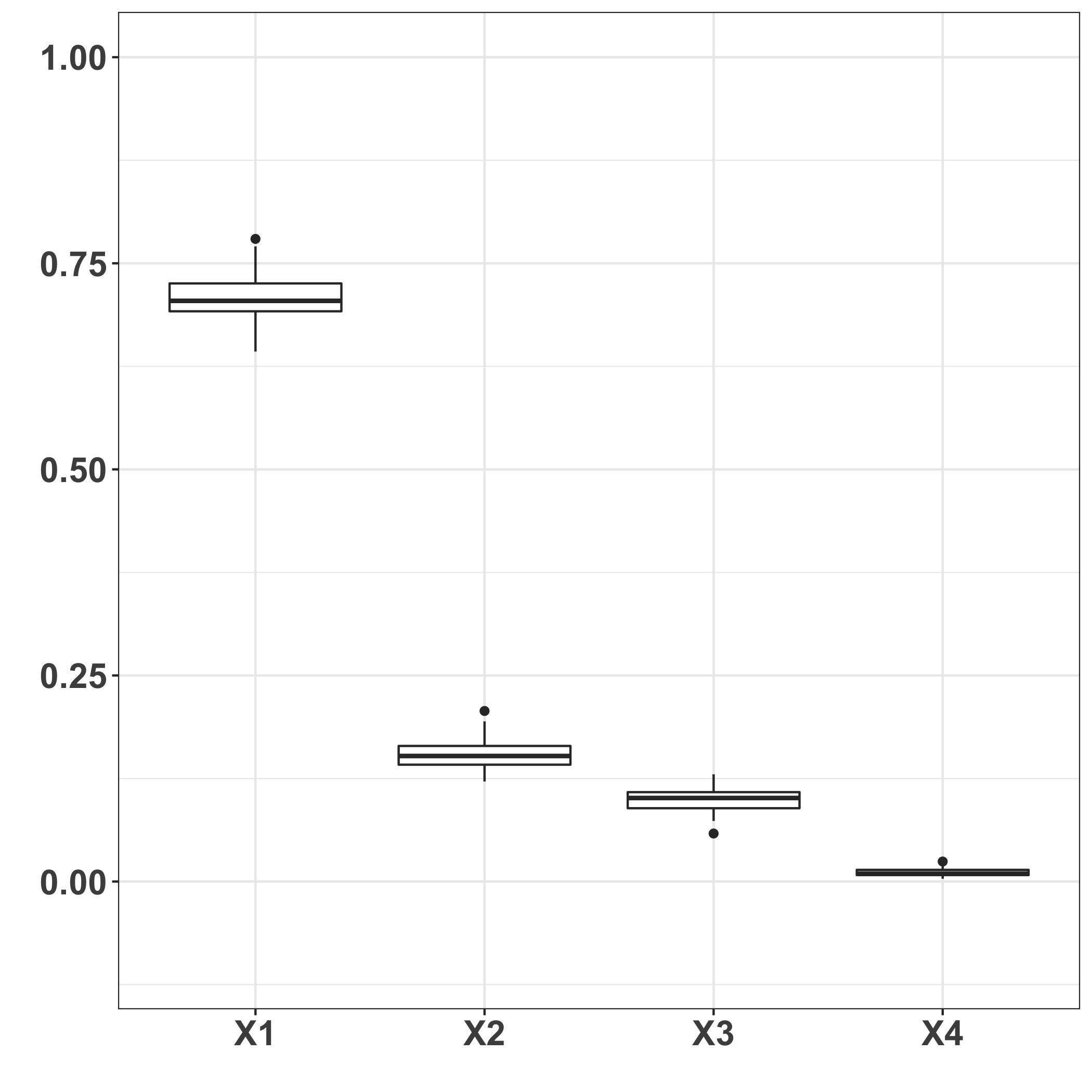}
         \caption{HSIC-based first-order index}
     \end{subfigure}
     \hfill
     \begin{subfigure}[b]{0.49\textwidth}
         \centering
         \includegraphics[width=\textwidth]{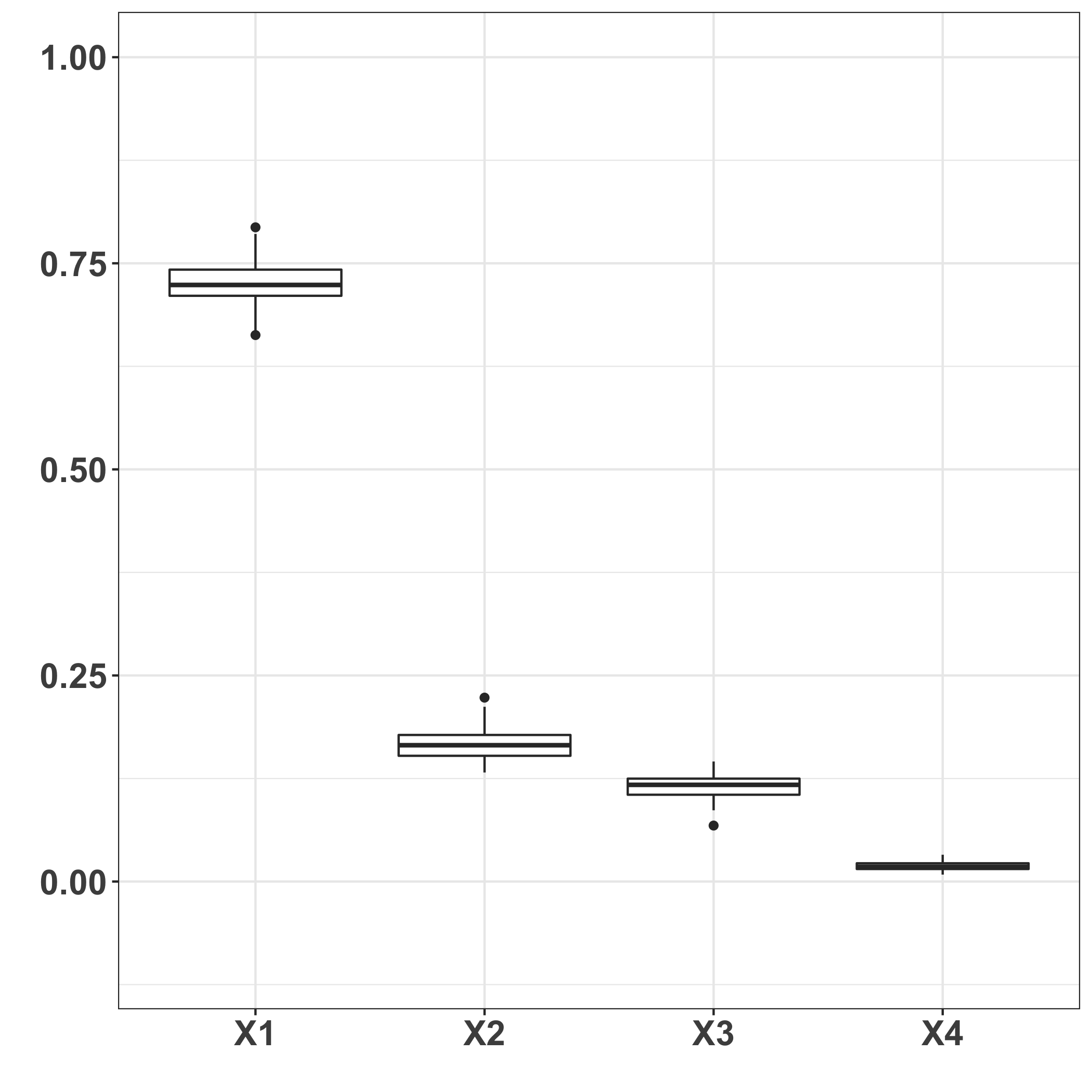}
         \caption{HSIC-based total index}
     \end{subfigure}
\caption{Ishigami test case. First-order (a) and total (b) HSIC-based indices with V-statistic estimator, $n=1000$, 50 replicates.}
\label{fig:example1_HSIC}
\end{figure}

\subsection{Stochastic simulator}

Our second illustration is a more original setting for GSA which consists of a stochastic simulator where the numerical model outputs a probability distribution, or rather a sample from a probability distribution in practice, for a fixed value of the input variables. Here we use a test case proposed in \cite{mou15} which involves five input variables and writes
\begin{equation*}
Y = (X_1+2X_2+U_1)\sin(3X_3-4X_4+N) + U_2 + 5X_5B + \sum_{i=1}^5 iX_i
\end{equation*}
where $X_1,\ldots,X_5\sim\mathcal{U}(0,1)$ are the input variables and $U_1\sim\mathcal{U}(0,1)$, $U_2\sim\mathcal{U}(1,2)$, $N\sim\mathcal{N}(0,1)$ and $B\sim\textup{Bernoulli}(1/2)$ are additional random variables which are responsible for the simulator stochasticity. Note that we modify the constant in front of $X_5B $ to lessen the effect of $X_5$ as compared to \cite{mou15}. An example of the output distribution for 20 random fixed values of the input variables obtained each time with a sample of size $100$ for the stochastic ones is given in Figure \ref{fig:example2_KDE}.\\

\begin{figure}[h!]
\centering
     \includegraphics[width=0.8\textwidth]{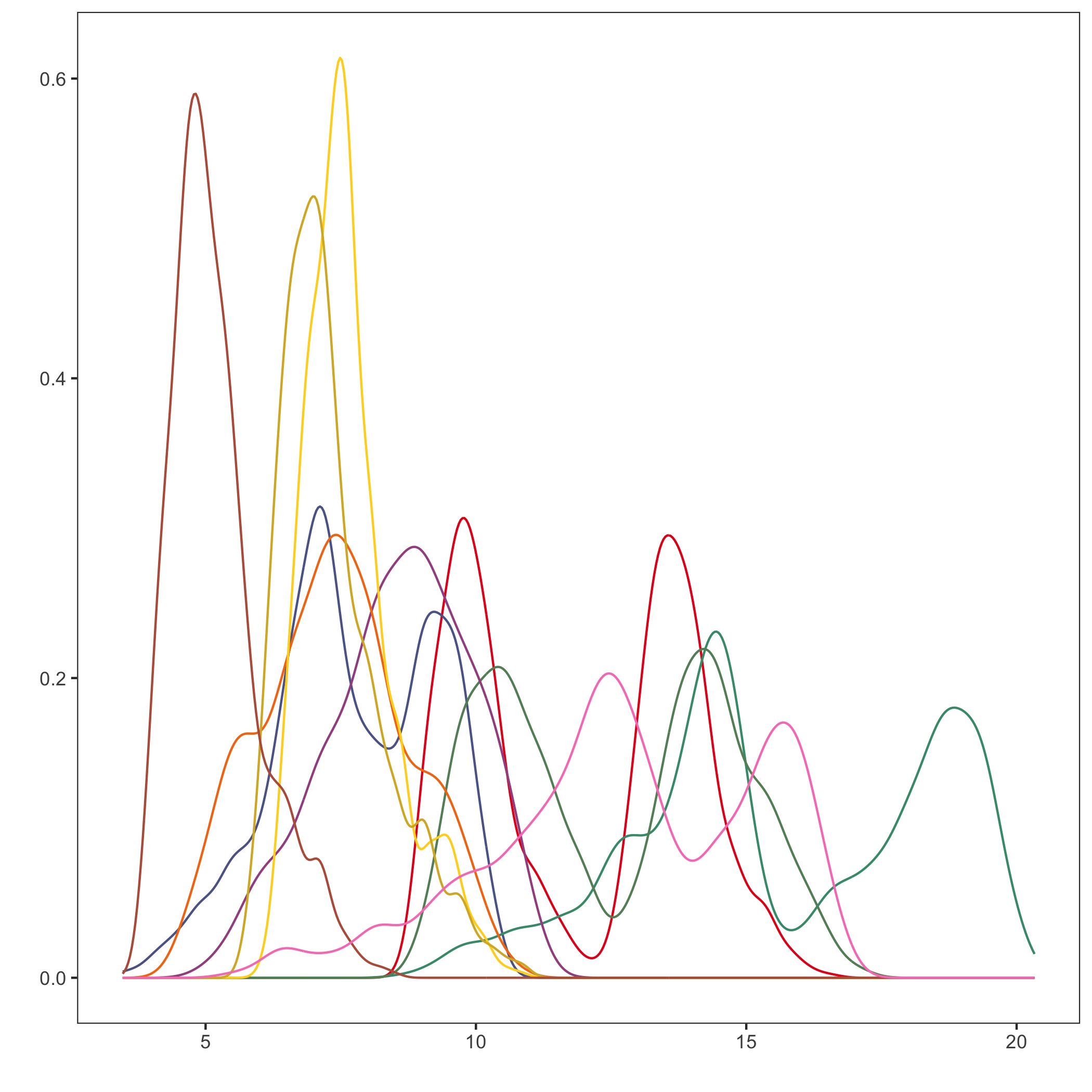}
\caption{Stochastic simulator test case. Output probability distribution for $20$ values of the input variables chosen at random. The distribution is estimated with a kernel-density estimator.}
\label{fig:example2_KDE}
\end{figure}

Leaving aside for now the whole output distribution, we first place ourselves in a standard GSA deterministic setting by first analyzing the input influence on both the output mean and standard deviation (with respect to $U_1$, $U_2$, $N$ and $B$).  We thus compute Sobol' indices for these two outputs of interest with a pick-freeze estimator with a sample of size $n=1000$ and perform $50$ replications, see Figure \ref{fig:example2_Sobol}. It shows that interactions are negligible, and that $X_5$ is clearly the most influential input by far: it explains alone $65\%$ of the output mean variability and $75\%$ of the output standard deviation variability. This is expected since $X_5$ is coupled with $B$, which creates the multi-modal feature of the output distribution. The output mean variability also depends on $X_3$ and $X_4$ to some lesser extent, and the output standard deviation variability on $X_2$.\\

\begin{figure}[h!]
\centering
\begin{subfigure}[b]{0.49\textwidth}
         \centering
         \includegraphics[width=\textwidth]{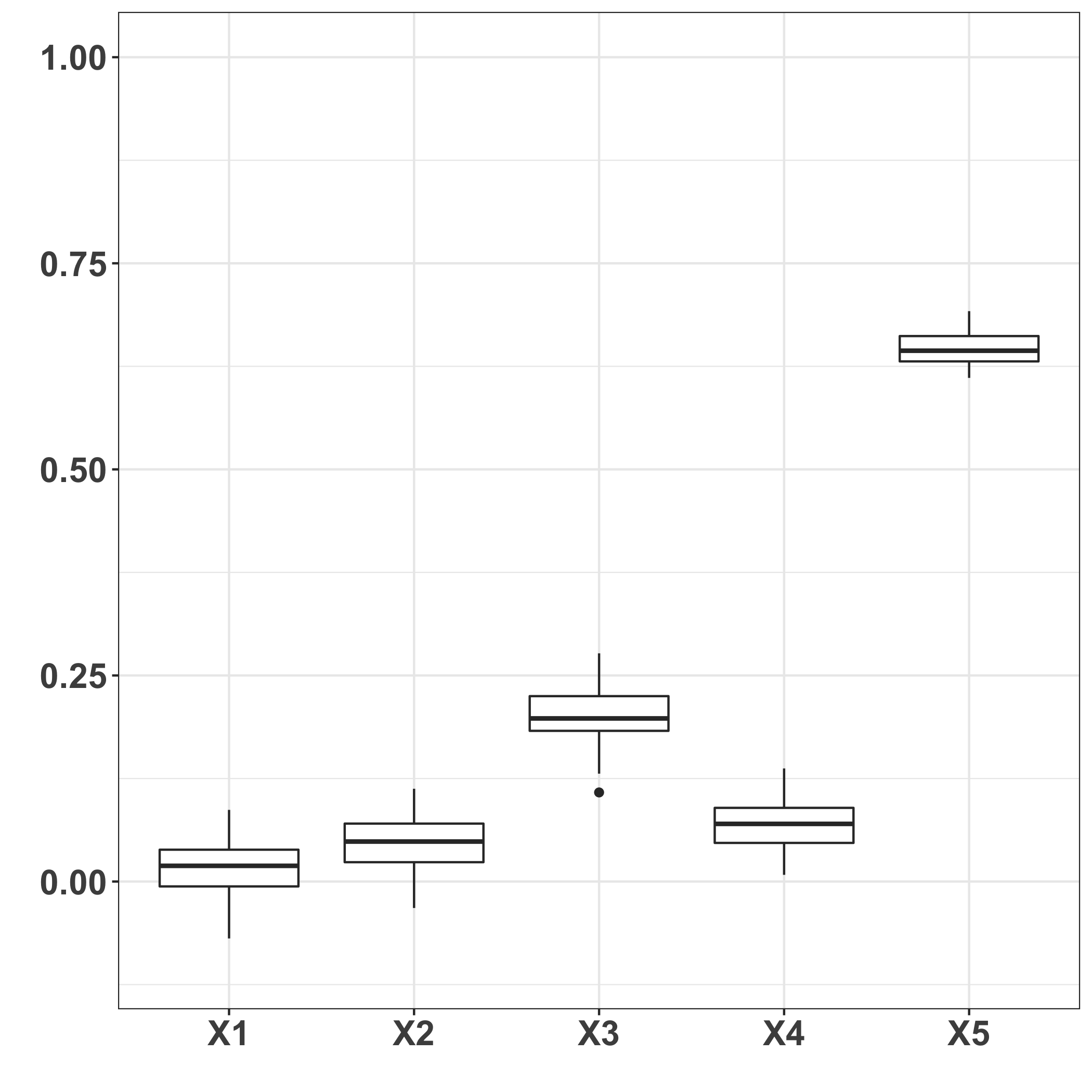}
         \caption{Sobol' first-order index of the output mean}
     \end{subfigure}
     \hfill
     \begin{subfigure}[b]{0.49\textwidth}
         \centering
         \includegraphics[width=\textwidth]{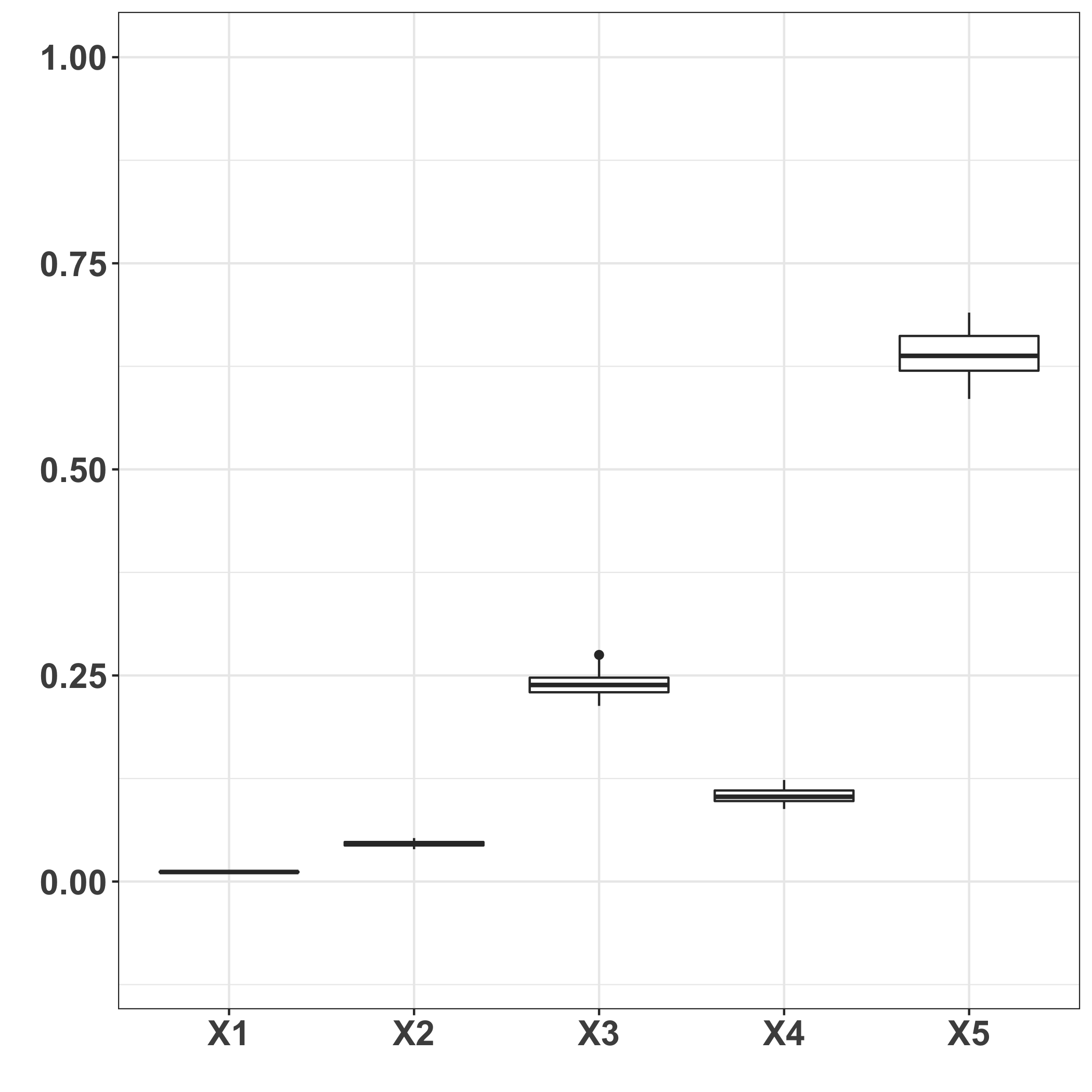}
         \caption{Sobol' total index of the output mean}
     \end{subfigure}
\begin{subfigure}[b]{0.49\textwidth}
         \centering
         \includegraphics[width=\textwidth]{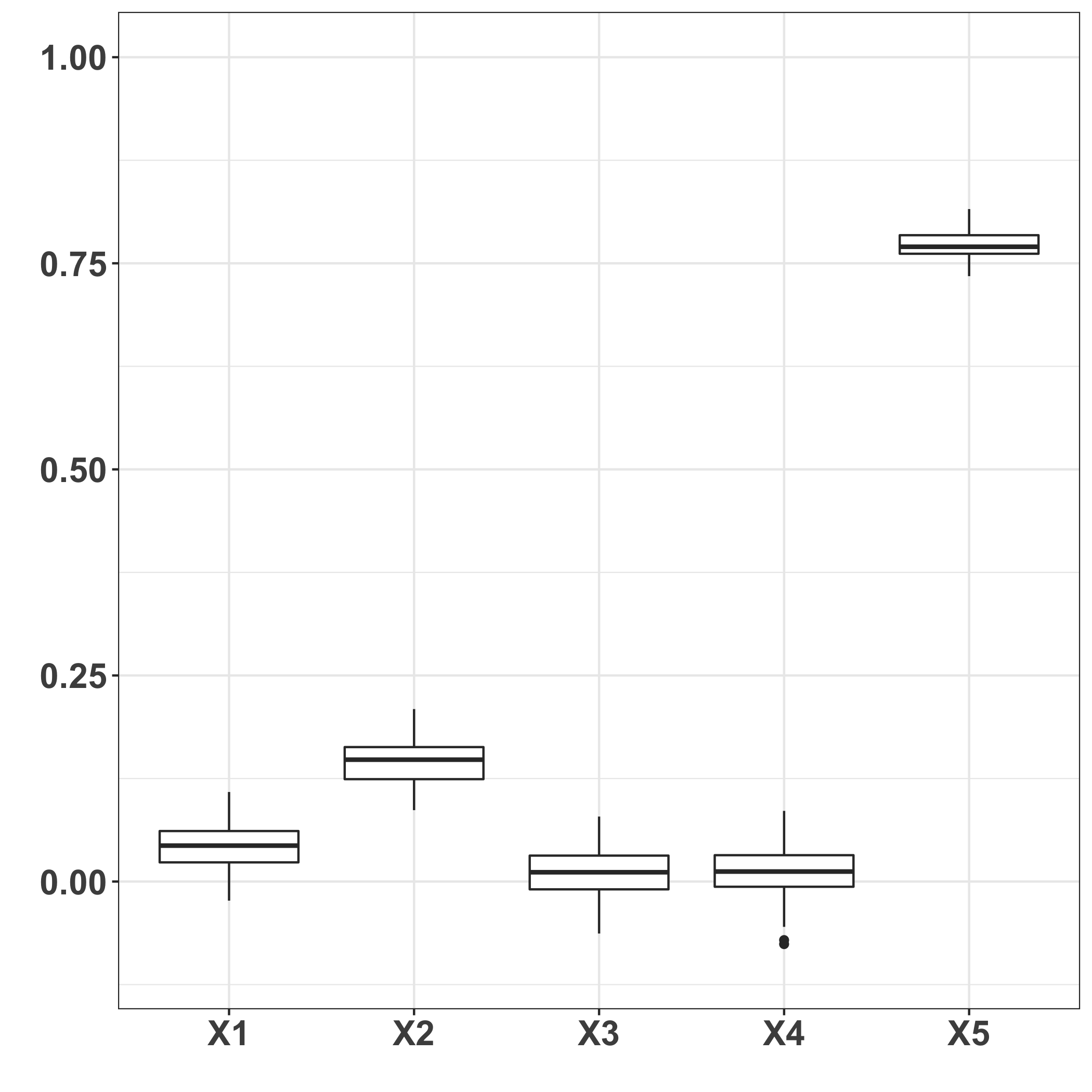}
         \caption{Sobol' first-order index of the output standard deviation}
     \end{subfigure}
     \hfill
     \begin{subfigure}[b]{0.49\textwidth}
         \centering
         \includegraphics[width=\textwidth]{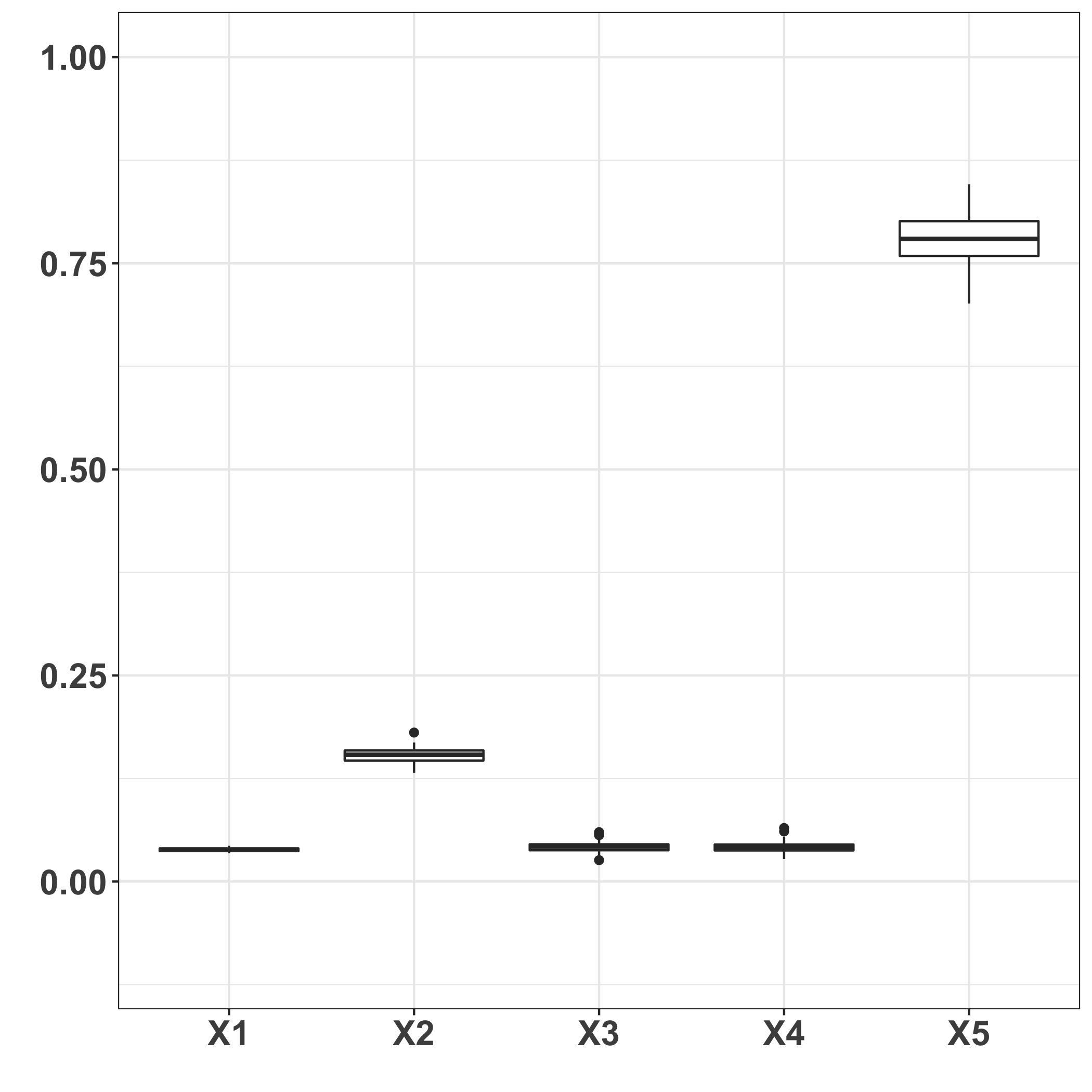}
         \caption{Sobol' total index of the output standard deviation}
     \end{subfigure}
\caption{Stochastic simulator test case. First-order (a) and total (b) Sobol' indices of the output mean and first-order (c) and total (d) Sobol' indices of the output standard deviation with pick-freeze estimators, $n=1000$, 50 replicates.} \label{fig:example2_Sobol}
\end{figure}

We now make use of the kernel framework to compute MMD- and HSIC-based sensitivity indices which can accommodate directly the output distribution thanks to the specific kernels discussed in Section \ref{sec:specific}. More precisely we use the kernel of Eq. (\ref{eq:kerneldistr}) with $\sigma^2=1$ and $\lambda$ chosen as the median of the $\M^2$ computed on the preliminary sample used for visualization in Figure \ref{fig:example2_KDE}, with a kernel $k_{\mathcal{Y}}(y,y') = \exp(-\frac{1}{2\tau^2}(y-y')^2)$  and $\tau$ chosen as the median of the pairwise distances between the output samples. We only compute first-order indices here and use for illustration the rank estimator of the MMD index from Section \ref{sec:mmdrank} while for HSIC we use again the Sobolev kernel. Results with $50$ replications and a sample of size $n=200$ are given in Figure \ref{fig:example2_MMDHSIC}. Both indices coincide and identify $X_5$ as the most important input variable, as well as a small influence of $X_3$, $X_2$ and $X_4$ while $X_1$ is non-important: considering the whole output distribution variability via the specific kernel is comparable to an aggregation of the variability on the output mean and standard deviation (and other moments we did not compute above).

\begin{figure}[h!]
\centering
\begin{subfigure}[b]{0.49\textwidth}
         \centering
         \includegraphics[width=\textwidth]{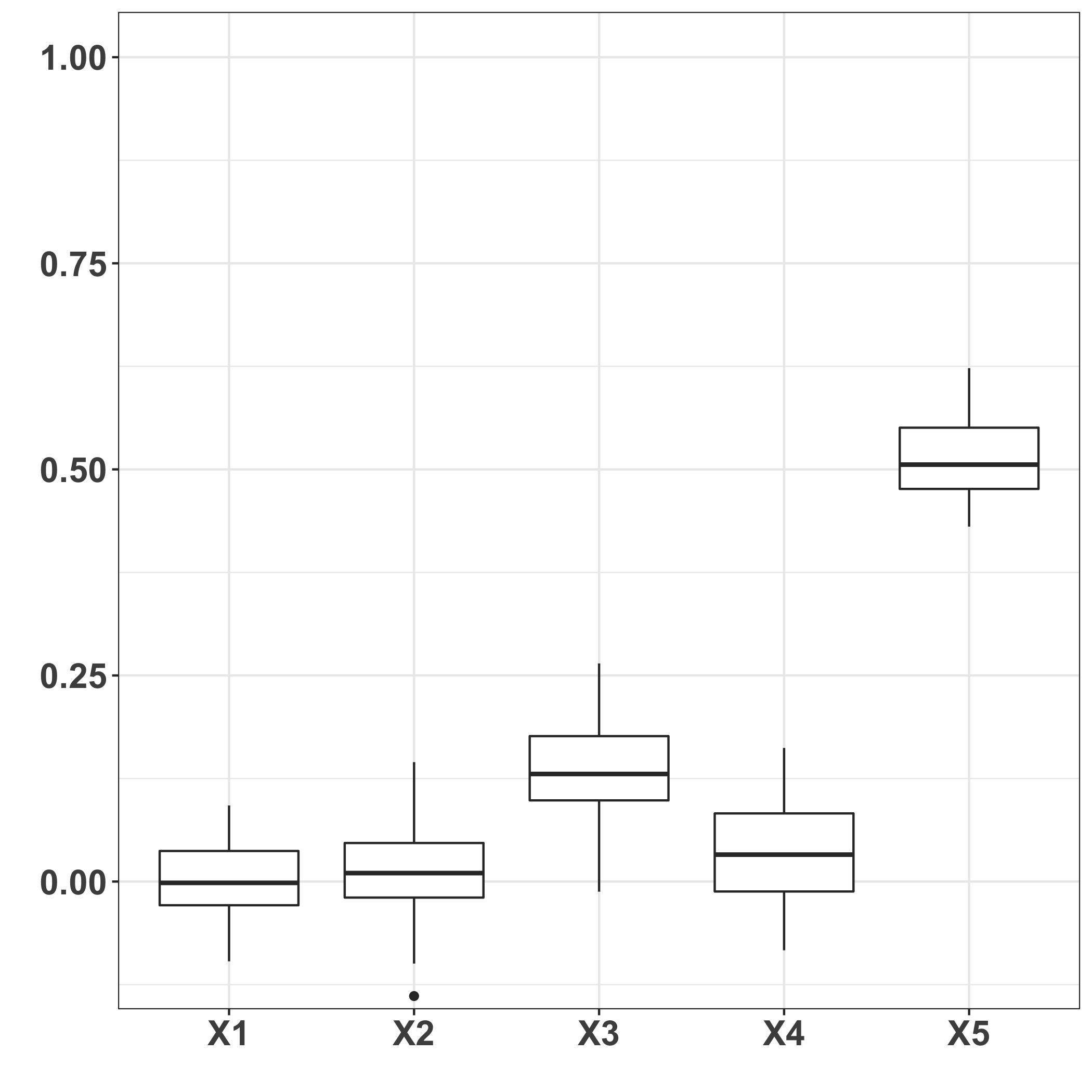}
         \caption{MMD first-order index}
     \end{subfigure}
     \hfill
     \begin{subfigure}[b]{0.49\textwidth}
         \centering
         \includegraphics[width=\textwidth]{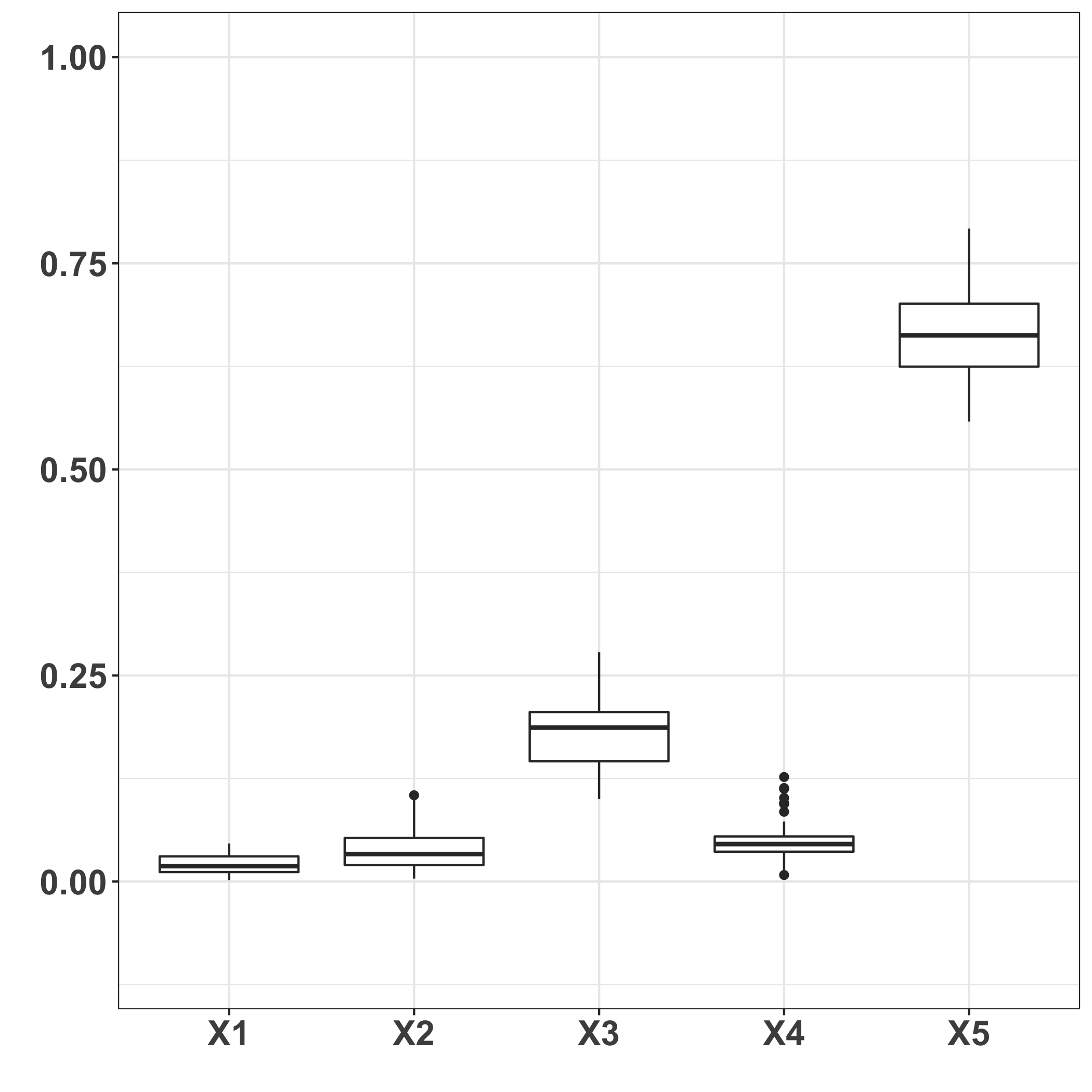}
         \caption{HSIC first-order index}
     \end{subfigure}
\caption{Stochastic simulator test case. First-order MMD (a) and HSIC (b) indices of the output distribution with rank and V-statistic estimators, respectively, $n=200$, 50 replicates.} \label{fig:example2_MMDHSIC}
\end{figure}

\subsection{Functional output}

Another commonly encountered industrial application is a physics-based numerical simulator involving functional outputs, such as curves representing the evolution over time of some system characteristics (\textit{e.g.} pressure, temperature, ...). To illustrate how time-series kernels can easily handle GSA on such systems we build a simplified compartmental epidemiological model inspired by previous works on COVID-19 \citep{magal20,char20,did20}. Our model is a straightforward Susceptible - Infected - Recovered (SIR) model \citep{ker27} which is slightly modified, in the sense that it accounts for two different types of infectious people: the reported cases, which we assume are isolated and can non longer contaminate others, and the unreported cases who can infect others. A summary of this compartment model proposed by \cite{magal20} is given in Figure \ref{fig:example3_sirmodel}. 

\begin{figure}[h!]
\centering
      \includegraphics[width=\textwidth]{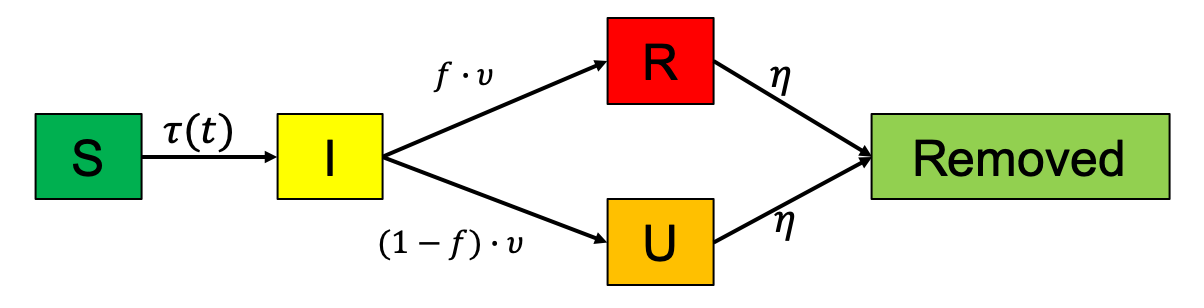}
\caption{Functional simulator test case. The modified SIR model with $4$ compartments following \cite{magal20} .} \label{fig:example3_sirmodel}
\end{figure}

$S$ consists of the susceptible individuals who are not yet infected. During the epidemic spread they are infected depending on the time-dependent transmission rate $\tau(t)$. Once infected they mode to compartment $I$ where are the asymptomatic infectious individuals. After a period of $\eta$ days they become symptomatic and a fraction $f$ of them is detected and go to compartment $R$, while the rest of them are undetected and go to compartment $U$. After a recovering period of $\eta$ days symptomatic people from $R$ and $U$ recover and go to the last compartment. Observe that this is a highly simplified representation of the epidemic where we do not account for hospitalizations, testing strategies or deaths: our goal here is not to be representative of COVID-19 but rather exemplify how GSA can be applied to such models.\\

The dynamics of the evolution of individuals from a compartment to another is modeled with the following system of ordinary differential equations:
\begin{eqnarray*}
\frac{dS}{dt} &=& - \tau S(I+U)\\
\frac{dI}{dt} &=& \tau S(I+U) - \nu I\\
\frac{dR}{dt} &=& f\nu I - \eta R\\
\frac{dU}{dt} &=& (1-f)\nu I - \eta U
\end{eqnarray*}
The transmission rate is chosen according to \cite{magal20} where they propose a parametric form given by $\tau(t) = \tau_0 \exp(-\mu \max(t-N,0))$. The underlying assumption is that before the epidemic outbreak the transmission rate is constant equal to $\tau_0$ and it then decreases with an exponential decay with rate $\mu$ once social distancing and lockdown start to have an effect after $N$ days. They further assume that the cumulative number of reported cases $CR(t)$ is approximately 
\begin{equation*}
CR(t) = \chi_1 \exp(\chi_2 t) - 1
\end{equation*}
where $\chi_1$ and $\chi_2$ are to be estimated on data. From this assumption they get the value of the initial conditions $I_0$, $U_0$, $R_0$
\begin{equation*}
I_0 = \frac{\chi_2}{f\nu},\ U_0 = \frac{(1-f)\nu}{\eta+\chi_2}I_0,\ R_0=1
\end{equation*}
and with in our case $S_0=66.99\times10^6$ is the initial susceptible population (here in France). From a GSA perspective we then assume that we have uncertainty on the following $6$ input variables: $\tau_0$, $\mu$, $N$ (transmission rate), $\eta$, $\nu$ (days until symptoms and recovery) and $\chi_2$ (which impacts the initial conditions). $f$ is assumed to be fixed at a fraction equal to $0.1$. We assign uniform distributions to the input variables with ranges consistent with the values from \cite{magal20}, \textit{i.e.} $\tau_0\sim\mathcal{U}(5.9\times10^{-9},6.1\times10^{-9})$, $\mu\sim\mathcal{U}(0.028,0.036)$, $N\sim\mathcal{U}(8,15)$, $1/\eta\sim\mathcal{U}(5,9)$, $1/\nu\sim\mathcal{U}(5,9)$ and $\chi_2\sim\mathcal{U}(0.32,0.4)$. An example of the dynamics of compartments $I$ and $R$ for $20$ values of the inputs chosen at random according to these uniform distributions is given in Figure \ref{fig:example3_Curves}.\\

\begin{figure}[h!]
\centering
\begin{subfigure}[b]{0.49\textwidth}
         \centering
         \includegraphics[width=\textwidth]{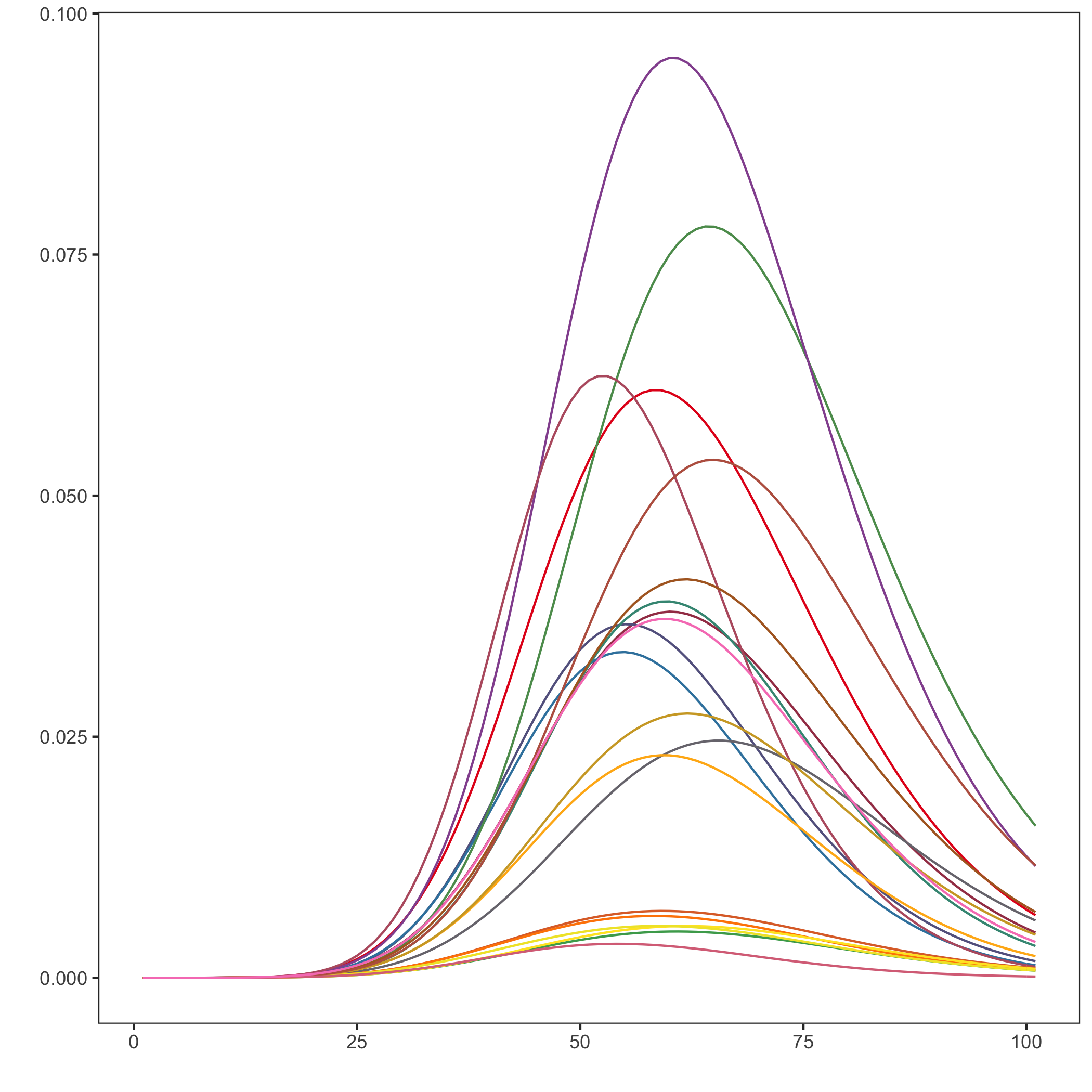}
         \caption{Infectious cases}
     \end{subfigure}
     \hfill
     \begin{subfigure}[b]{0.49\textwidth}
         \centering
         \includegraphics[width=\textwidth]{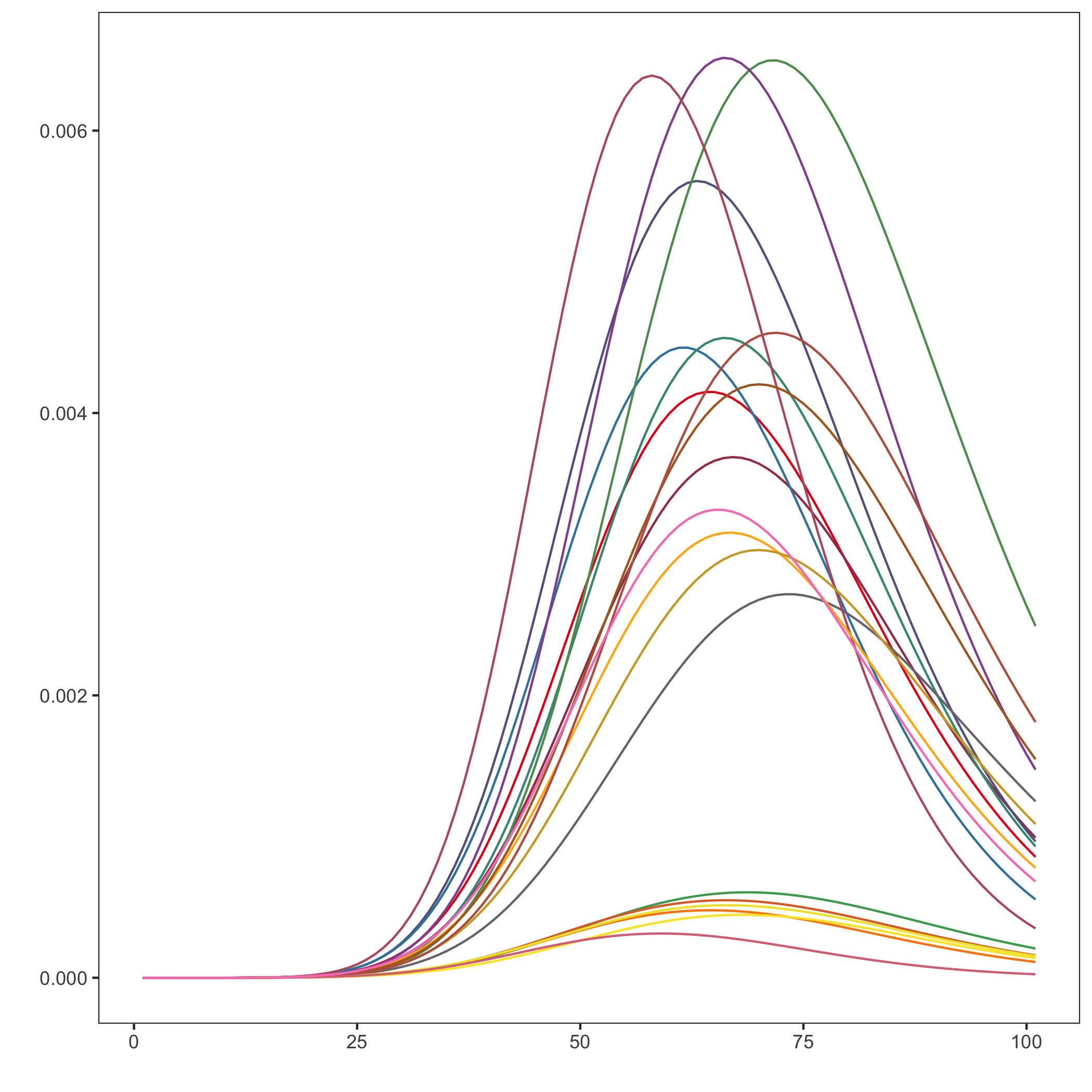}
         \caption{Reported cases}
     \end{subfigure}
\caption{Functional simulator test case. Output dynamics over time for compartment $I$ (left) and $R$ (right) for $20$ values of the input variables chosen at random. They are both normalized by the total population $S_0$.} \label{fig:example3_Curves}
\end{figure}

For GSA we rely on the global-alignment kernel of \cite{cut11} designed for time-series, which searches for all their alignments and averages them, and use it inside our first-order HSIC indices for the output, whereas we still employ the Sobolev kernel for the inputs. The results obtained with $50$ repetitions with a sample of size $n=200$ and the V-statistic estimator are reported in Figure \ref{fig:example3_HSIC}.

\begin{figure}[h!]
\centering
\begin{subfigure}[b]{0.49\textwidth}
         \centering
         \includegraphics[width=\textwidth]{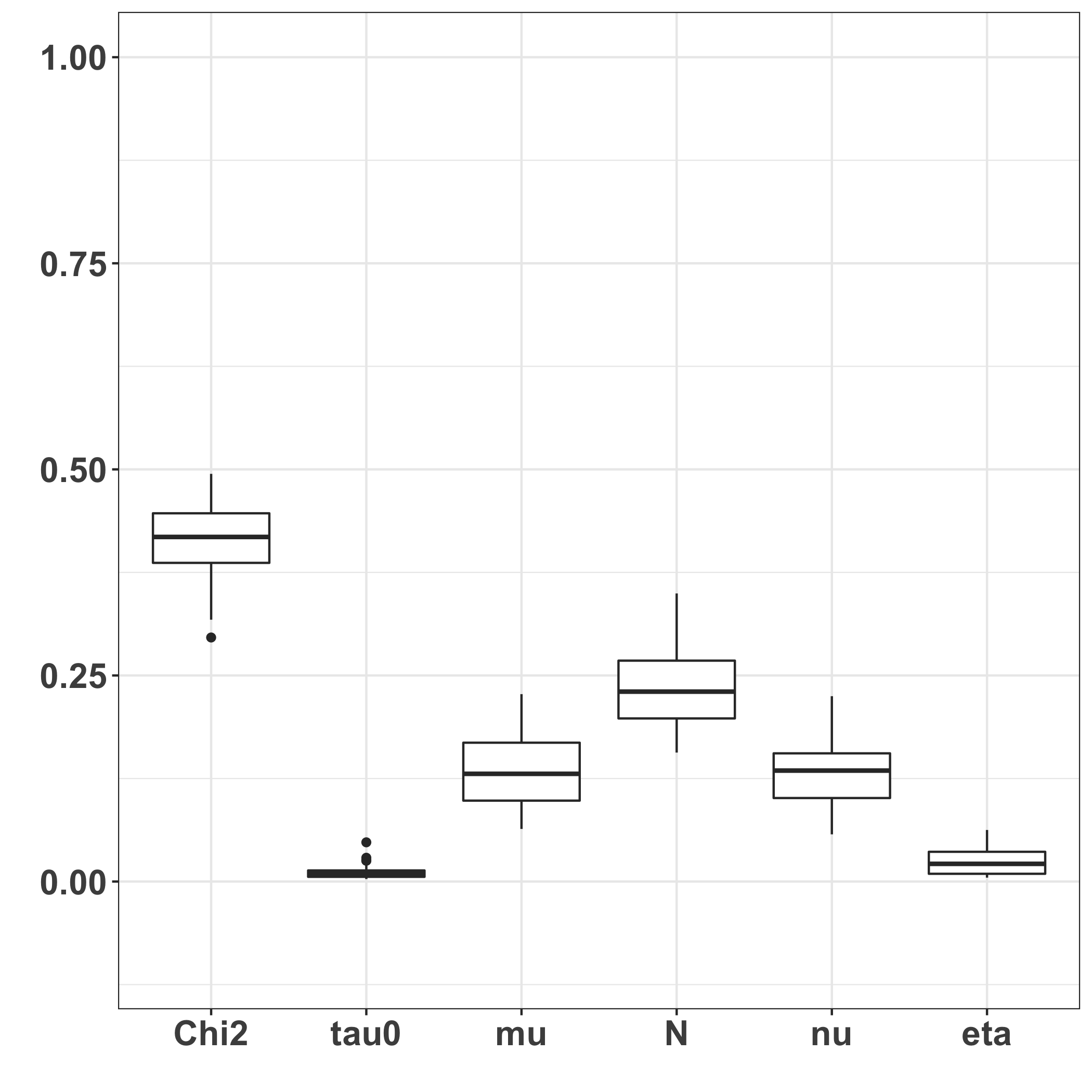}
         \caption{First-order HSIC index for compartment $I$}
     \end{subfigure}
     \hfill
     \begin{subfigure}[b]{0.49\textwidth}
         \centering
         \includegraphics[width=\textwidth]{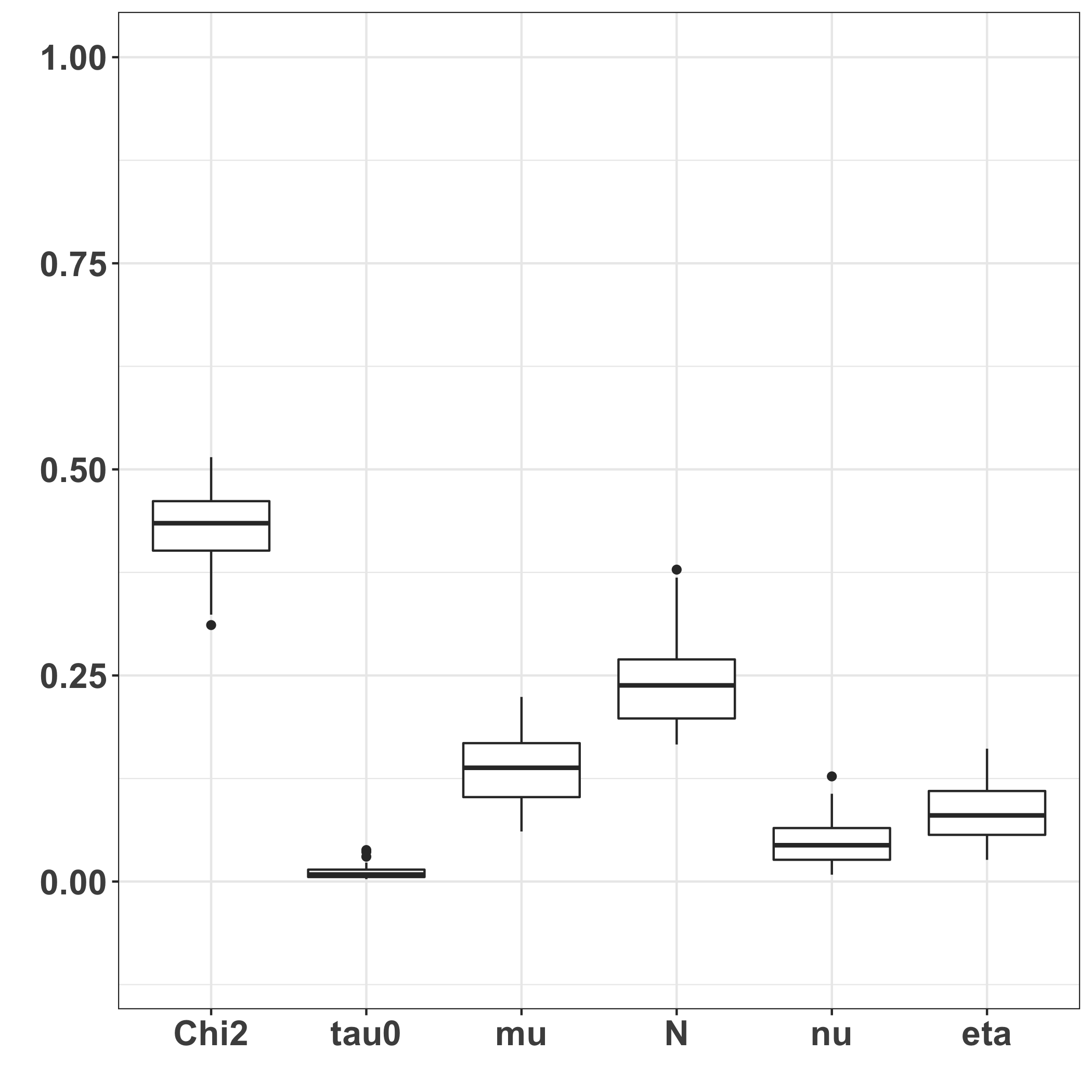}
         \caption{First-order HSIC index for compartment $R$}
     \end{subfigure}
\caption{Functional simulator test case. First-order HSIC index for compartments $I$ (left) and $R$ (right) with V-statistics estimator, $n=200$, $50$ replicates.} \label{fig:example3_HSIC}
\end{figure}

For both compartments the most influential input is $\chi_2$, as expected since it influences the initial conditions, and then $N$ and $\mu$ related to the transmission rate. $\nu$ has also an impact for compartment $I$ but not $\eta$, which is coherent with the ordinary differential equations, and one can see on the contrary that $\eta$ influences compartment $R$.

\subsection{Multi-class output with dependent inputs}

Finally we investigate a numerical model with both a categorical output (to make use of the discussion from Section \ref{sec:specific}) and dependent inputs (to analyze kernel-embedding Shapley effects from Section \ref{sec:shapley}). We build upon the famous wine quality data set \citep{cortez09} of the UCI repository \citep{dua17}. This dataset consists of $4898$ observations of wine qualities (categorical variable with levels $0$ to $10$ corresponding to a score) associated to $11$ features obtained with physicochemical tests. In order to place ourselves in a standard computer experiments setting (\textit{i.e.} a numerical simulator and uncertain inputs with given probability distribution) we use this dataset to design a GSA scenario detailed in the following steps:
\begin{enumerate}
\item We regroup wine quality scores into only $3$ categories: low (score less than $5$), medium (score equal to $6$) and high (score higher than $7$) in order to have a balanced dataset. We also use a small subsample of size $600$ of white wine only from the initial $4898$ observations for faster estimation of the input dependence structure;
\item We estimate a random forest model between the wine quality and the $11$ inputs from this transformed dataset and compute variable importance for each input. The variable importance score is used to select only $4$ important features among the initial $11$ ones (volatile acidity, chlorides, density and alcohol). This is absolutely not a mandatory step, but we choose to do so for both a faster computation of Shapley effects and estimation of the input dependence structure. A new random forest model is finally built with these $4$ input variables only, and the predictor serves as our numerical simulation model;
\item The samples from the $4$ input variables identified above are used to estimate a vine copula structure which models their dependence \citep{czado19}. Once the vine copula is estimated, it is then easy to generate new input samples as much as required.\\
\end{enumerate}

MMD- and HSIC-Shapley effects are then computed with a sample size of $n=1000$ with a dirac categorical kernel for the output and a Sobolev kernel for the inputs in the HSIC case. For MMD we use the nearest-neighbor estimator of Section \ref{sec:shapleyestim} and for HSIC the V-statistic estimator and we repeat the estimation $50$ times, see Figure \ref{fig:example4_Shapley}.

\begin{figure}[h!]
\centering
\begin{subfigure}[b]{0.49\textwidth}
         \centering
         \includegraphics[width=\textwidth]{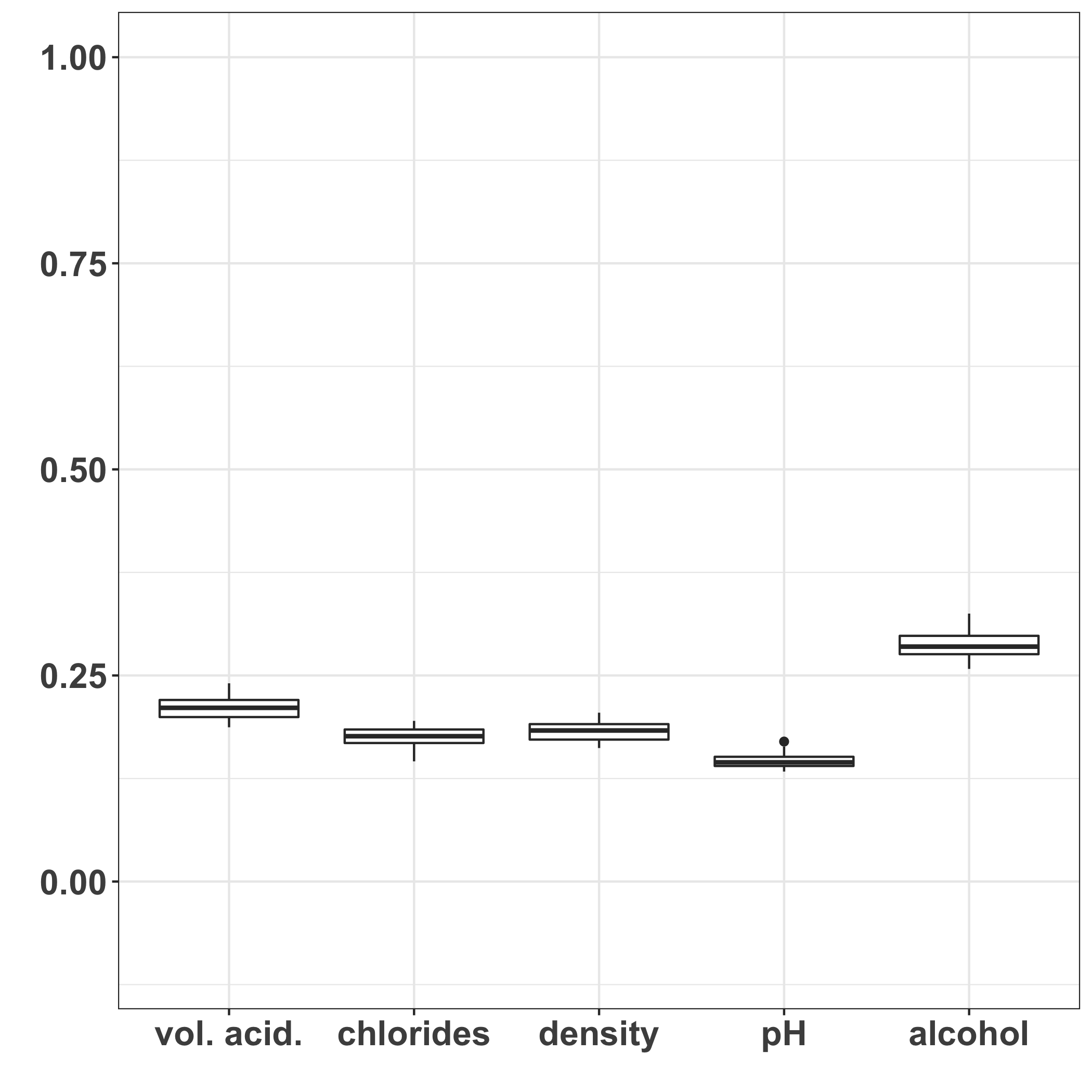}
         \caption{MMD-Shapley effect}
     \end{subfigure}
     \hfill
     \begin{subfigure}[b]{0.49\textwidth}
         \centering
         \includegraphics[width=\textwidth]{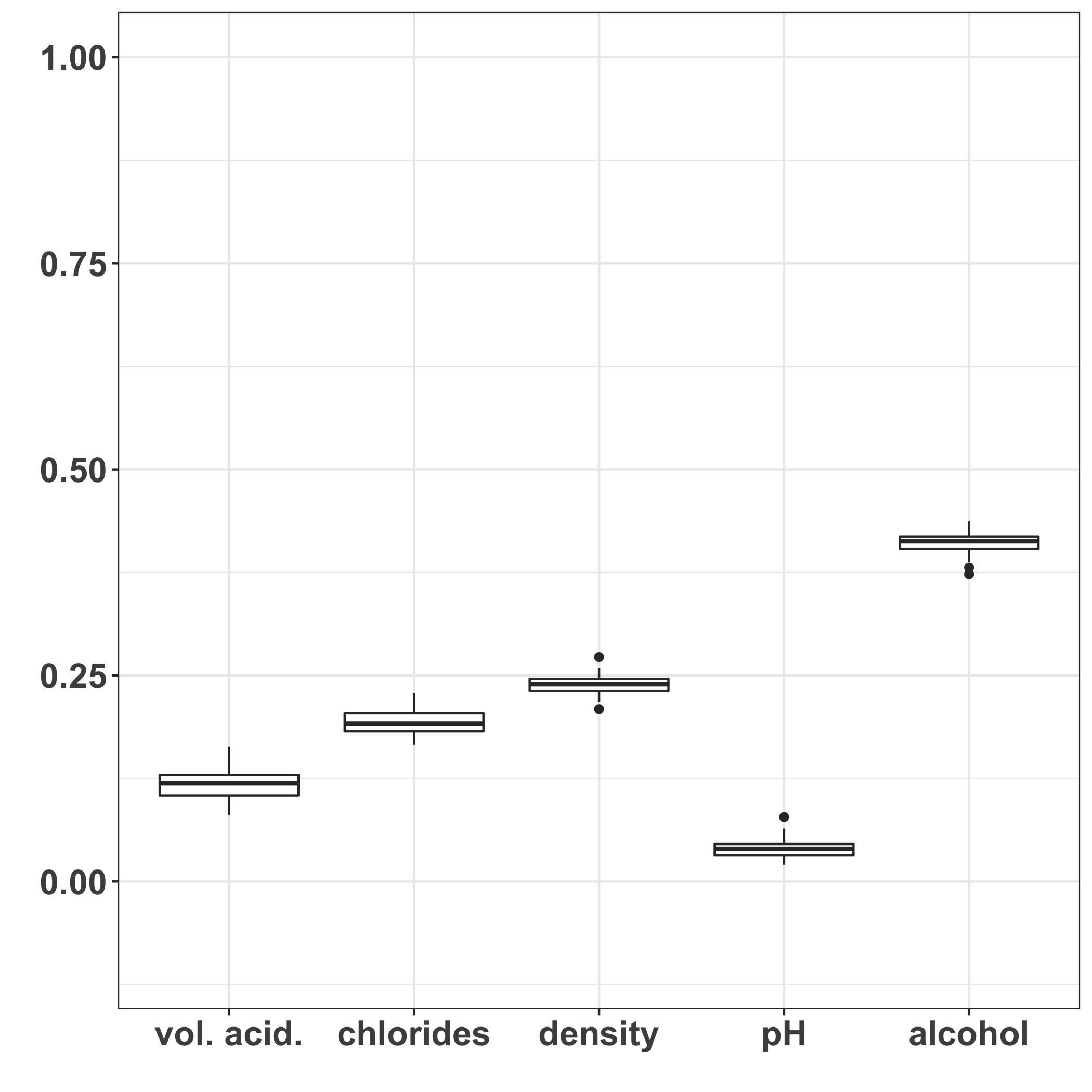}
         \caption{HSIC-Shapley effect}
     \end{subfigure}
\caption{Multi-class output  test case. MMD- (a) and HSIC- (b) Shapley effects with nearest-neighbor and V-statistic estimators, respectively, $n=1000$, 50 replicates.} \label{fig:example4_Shapley}
\end{figure}

Both kernel-embedding Shapley effects identify alcohol as the most influential input, which was expected from the variable importance scores computed with the random forest. However, the MMD-Shapley effects do not discriminate as clearly the input variables as HSIC. We suspect that there may be remaining estimation bias coming from the nearest-neighbor estimators which we plan to carefully examine in future work.

\section{Conclusion}
In this paper we discussed two moment-independent sensitivity indices which generalize Sobol' ones by relying on the RKHS embedding of probability distributions. These MMD- and HSIC-based sensitivity indices are shown to admit an ANOVA-decomposition, which makes it possible to properly define input interactions and their natural normalization constant. To the best of our knowledge this is the first time such a result is proved for sensitivity indices apart from Sobol' ones. We also defined kernel-embedding Shapley effects which are built upon these indices for the case where the input variables are no longer independent. As discussed through several GSA applications with categorical outputs or stochastic simulators, this opens the path for new powerful and general GSA approaches by means of kernels adapted to the task at hand. Finally, several estimators have been introduced, including new ones inspired by recent advances in Sobol' indices and Shapley effects estimation.\\
However, there is still room for improvement in the theoretical understanding of theses indices. First, we extensively used Mercer's theorem and it would be interesting to extend our results when it no longer holds. We also assume a kernel product form for HSIC indices, whereas the theorem used in our proof allows for more general kernels. From an estimation perspective, we did not exhibit here any central limit theorem, although this would be an important step enabling to statistically test whether indices are zero or not. But this is not at all an easy task, which may be tackled via the functional delta method combined with Mercer's theorem. On the other hand, some bias can be observed in the nearest neighbor estimators, which should be analyzed carefully in future work. Finally, further practical experimentations should be performed to better understand the behavior of these new indices. We think in particular to the choice of the kernel hyperparameters, and the investigation of invariant kernels for outputs given as curves or images.

\appendix

\section{Proofs}

\subsection{Proof of Theorem \ref{th:MMDanova}} \label{sec:MMDanova}

\begin{proof}
The theorem is proved in the case where Mercer's theorem holds, \textit{i.e.}, the output is assumed to be such that $Y\in\mathcal{Y}$ with $\mathcal{Y}$ a compact set and $k_{\mathcal{Y}}$ has the representation 
\begin{equation*}
k_{\mathcal{Y}}(y,y') = \sum_{r=1}^{\infty} \phi_r(y)\phi_r(y')
\end{equation*}
as in Eq. (\ref{eq:MMDindexgroup}). Consider now the random variable $W=\sum_{r=1}^{\infty} \eta^{[r]}({\bX})$ where $\eta^{[r]}({\bX}) = \phi_r(Y) = \phi_r(\eta({\bX}))$. 
To prove the theorem, two formulations of $\var\, W$ are exhibited. 
First, since the functions $\phi_r$ are orthogonal in $\mathbb{L}^2(\mathcal{Y})$ and using the absolute convergence of the series, we have 
\begin{eqnarray*}
\var\, W &=& \sum_{r=1}^{\infty} \var\, \phi_r(Y)\\
&=& \sum_{r=1}^{\infty} \EE\left(\phi_r(Y)\phi_r(Y)\right) - \sum_{r=1}^{\infty}\EE\left(\phi_r(Y)\phi_r(Y')\right)\\
&=& \EE\left(\sum_{r=1}^{\infty}\phi_r(Y)\phi_r(Y)\right) -\EE\left(\sum_{r=1}^{\infty}\phi_r(Y)\phi_r(Y')\right)\\
&=& \EE k(Y,Y) -\EE k(Y,Y').
\end{eqnarray*}
On the other hand, using the variance decomposition (\ref{th:anova}) for each $\eta^{[r]}({\bX}) = \phi_r(\eta({\bX}))$ we get
\begin{eqnarray*}
\var\, W &=& \sum_{r=1}^{\infty} \var\, \eta^{[r]}({\bX})\\
&=& \sum_{r=1}^{\infty} \sum_{A\subseteq\Pd} \sum_{B \subset A}(-1)^{|A|-|B|} \var\, \EE\left(\eta^{[r]}({\bX})\vert{\bX}_B\right) \\
&=& \sum_{A\subseteq\Pd} \sum_{B \subset A}(-1)^{|A|-|B|} \sum_{r=1}^{\infty} \var\, \EE\left(\phi_r(Y)\vert{\bX}_B\right) \\
&=& \sum_{A\subseteq\Pd} \sum_{B \subset A}(-1)^{|A|-|B|}  \EE_{{\bX}_B}\left(\M^2(\textup{P}_{Y},\textup{P}_{Y \vert {\bX}_B})\right)
\end{eqnarray*}
using again the absolute continuity and the expansion of $\EE_{{\bX}_B}\left(\textup{MMD}^2(\textup{P}_{Y},\textup{P}_{Y | {\bX}_B})\right)$ obtained in Eq. (\ref{eq:MMDindexgroup}). The theorem follows by equating both formulations of $\var\, W$.
\end{proof}

\subsection{Proof of Proposition \ref{prop:lawtotalvar}} \label{sec:lawtotalvar}

\begin{proof}
We simply add and subtract $\EE_{{\bX}_A} \EE_{\xi,\xi'\sim\textup{P}_{Y \vert {\bX}_A}} k_{\mathcal{Y}}(\xi,\xi')$:
\begin{eqnarray*}
\M^2_{\textup{tot}} &=& \EE_{\zeta\sim\textup{P}_Y} k_{\mathcal{Y}}(\zeta,\zeta) - \EE_{\zeta,\zeta'\sim\textup{P}_Y} k_{\mathcal{Y}}(\zeta,\zeta') \\
&=& \EE_{\zeta\sim\textup{P}_Y} k_{\mathcal{Y}}(\zeta,\zeta) - \EE_{{\bX}_A} \EE_{\xi,\xi'\sim\textup{P}_{Y \vert {\bX}_A}} k_{\mathcal{Y}}(\xi,\xi') + \EE_{{\bX}_A}\left(\M^2(\textup{P}_{Y},\textup{P}_{Y \vert {\bX}_A})\right)\\
&=& \EE_{{\bX}_A}\EE_{\xi\sim\textup{P}_{Y \vert {\bX}_A}} k_{\mathcal{Y}}(\xi,\xi) - \EE_{{\bX}_A} \EE_{\xi,\xi'\sim\textup{P}_{Y \vert {\bX}_A}} k_{\mathcal{Y}}(\xi,\xi') + \EE_{{\bX}_A}\left(\M^2(\textup{P}_{Y},\textup{P}_{Y \vert {\bX}_A})\right)\\
&=& \EE_{{\bX}_A}\left[\EE_{\xi\sim\textup{P}_{Y \vert {\bX}_A}} k_{\mathcal{Y}}(\xi,\xi) - \EE_{\xi,\xi'\sim\textup{P}_{Y \vert {\bX}_A}} k_{\mathcal{Y}}(\xi,\xi')\right] + \EE_{{\bX}_A}\left(\M^2(\textup{P}_{Y},\textup{P}_{Y \vert {\bX}_A})\right).
\end{eqnarray*}
\end{proof}

\subsection{Proof of Theorem \ref{th:HSICanova}} \label{sec:HSICanova}

\begin{proof}
We first rewrite HSIC between $\bX$ and $Y$ from Eq. (\ref{eq:hsic}) as a multivariate integral, assuming $\textup{P}_{\bX Y}$ is absolutely continuous with respect to the Lebesgue measure on $\mathcal{X}\times\mathcal{Y}$:
\begin{eqnarray}
\HS(\bX,Y)&=&\int_{\mathcal{X}\times\mathcal{X}}\int_{\mathcal{Y}\times\mathcal{Y}} k_{\mathcal{X}}(\bx,\bx')k_{\mathcal{Y}}(y,y') \left[p_{\bX Y}(\bx,y)-p_{\bX}(\bx)p_{Y}(y)\right]\nonumber \\
&&\left[p_{\bX Y}(\bx',y')-p_{\bX}(\bx')p_{Y}(y')\right]d\bx d\bx' dy dy' \label{eq:inthsic}
\end{eqnarray}
where $p_{\bX Y}$, $p_{\bX}$ and $p_{Y}$ are the probability density functions of $(\bX,Y)$, $\bX$ and $Y$, respectively.
As in Theorem \ref{th:MMDanova} we further assume Mercer's theorem holds, which means that 
\begin{equation*}
k_{\mathcal{Y}}(y,y') = \sum_{r=1}^{\infty} \phi_r(y)\phi_r(y').
\end{equation*}
For each $r$, we then define the function
\begin{equation*}
g^{[r]}(\bx) = \int_{\mathcal{X}} \int_{\mathcal{Y}} k_{\mathcal{X}}(\bx,\bx')\phi_r(y') \left[p_{\bX Y}(\bx',y')-p_{\bX}(\bx')p_{Y}(y')\right]d\bx'dy',
\end{equation*}
noting that $g^{[r]}\in\F$ from Assumption \ref{as:finitekernelxy}. It is then straightforward to show that
\begin{eqnarray*}
\Vert g^{[r]}\Vert_{\F}^2 &= &\int_{\mathcal{X}\times\mathcal{X}}\int_{\mathcal{Y}\times\mathcal{Y}} k_{\mathcal{X}}(\bx,\bx')\phi_r(y)\phi_r(y') \left[p_{\bX Y}(\bx,y)-p_{\bX}(\bx)p_{Y}(y)\right]\nonumber \\
&&\left[p_{\bX Y}(\bx',y')-p_{\bX}(\bx')p_{Y}(y')\right]d\bx d\bx' dy dy',
\end{eqnarray*}
which means that
\begin{equation}
\HS(\bX,Y) = \sum_{r=1}^{\infty} \Vert g^{[r]}\Vert_{\F}^2 \label{eq:sumhsic}
\end{equation}
since in Mercer's theorem we have the absolute convergence of the series. Now the idea is to write an orthogonal decomposition (in $\F$) for each function $g^{[r]}$, which will finally provide a decomposition for HSIC through Eq. (\ref{eq:sumhsic}).\\

The orthogonal decomposition of  $g^{[r]}$ is obtained with Theorem 4.1 from \cite{kuo10}. First, recall that we have from the first part of Assumption \ref{as:HSIC}:
\begin{equation}
k_\mathcal{X}(\bx,\bx') = \prod_{l=1}^p \left(1 + k_l(x_l,x_l')\right) = \sum_{A\subseteq\Pd}\prod_{l\in A} k_l(x_l,x_l') := \sum_{A\subseteq\Pd} k_{A}(\bx_A,\bx_A') \label{eq:sumkernelnull}
\end{equation}
which corresponds to Eq. (4.1) in \cite{kuo10}. We then introduce a set of commuting projections $\left\{P_l\right\}_{l=1}^p$ on $\F$ given by
\begin{equation}
P_l(f) = \int_{\mathcal{X}_l} f(x_1,\ldots,x_{l-1},t,x_{l+1},\ldots,x_d)p_{X_l}(t)dt
\end{equation}
for all $f\in\F$. From the second part of Assumption \ref{as:HSIC}, one has for all subset $A\subseteq\Pd$ and $\bx_A\in\mathcal{X}_A$
\begin{equation*}
P_l\left(k_{A}(\cdot,\bx_A)\right)=\prod_{l'\in A,l'\neq l}k_{l'}(\cdot,x_{l'}) \int_{\mathcal{X}_l} k_l(t,x_{l})p_{X_l}(t)dt = 0
\end{equation*}
for $l\in A$, meaning that Eq. (4.5) from \cite{kuo10} is satisfied. From Theorem 4.1 from \cite{kuo10}, we can now state that $g^{[r]}(\bx)$ has an unique orthogonal decomposition given by
\begin{equation*}
g^{[r]} = \sum_{A\subseteq\Pd} g_A^{[r]}
\end{equation*}
where 
\begin{equation*}
g_A^{[r]} = \sum_{B\subseteq A} (-1)^{\vert A\vert - \vert B\vert} P_{-B}(g^{[r]})
\end{equation*}
with $P_{-B}=\prod_{l\notin B}P_l$. Since the decomposition is orthogonal, we further have
\begin{equation*}
\Vert g^{[r]}\Vert_{\F}^2 = \sum_{A\subseteq\Pd} \Vert g_A^{[r]}\Vert_{\F}^2
\end{equation*}
and
\begin{equation*}
\Vert g_A^{[r]}\Vert_{\F}^2 = \sum_{B\subseteq A} (-1)^{\vert A\vert - \vert B\vert} \Vert P_{-B}(g^{[r]}) \Vert_{\F}^2.
\end{equation*}
The last part is to expand $\Vert P_{-B}(g^{[r]}) \Vert_{\F}^2$. We first write the projection:
\begin{eqnarray*}
P_{-B}(g^{[r]})  &=& \int_{\mathcal{X}} \int_{\mathcal{Y}} \int_{\mathcal{X}_{-B}} k_{\mathcal{X}}(\bx,\bx')p_{\bX_{-B}}(\bx_{-B}) \phi_r(y') \left[p_{\bX Y}(\bx',y')-p_{\bX}(\bx')p_{Y}(y')\right]d\bx'dy'd\bx_{-B} \\
&=& \int_{\mathcal{X}} \int_{\mathcal{Y}}  \left(\prod_{l\notin B} \int_{\mathcal{X}_{l}} \left(1+k_l(x_l,x_l')\right) p_{X_l}(x_l)dx_l\right) \prod_{l\in B} \left(1+k_l(x_l,x_l')\right)\phi_r(y')\\
&& \left[p_{\bX Y}(\bx',y')-p_{\bX}(\bx')p_{Y}(y')\right]d\bx'dy' \\ 
&=& \int_{\mathcal{X}} \int_{\mathcal{Y}} \prod_{l\in B} \left(1+k_l(x_l,x_l')\right)\phi_r(y')\left[p_{\bX Y}(\bx',y')-p_{\bX}(\bx')p_{Y}(y')\right]d\bx'dy' \\ 
&=& \int_{\mathcal{X}_B} \int_{\mathcal{X}_{-B}} \int_{\mathcal{Y}}\prod_{l\in B} \left(1+k_l(x_l,x_l')\right)\phi_r(y')\left[p_{\bX Y}(\bx',y')-p_{\bX}(\bx')p_{Y}(y')\right]d\bx'_Bd\bx'_{-B}dy' \\ 
&=& \int_{\mathcal{X}_B} \int_{\mathcal{Y}}\prod_{l\in B} \left(1+k_l(x_l,x_l')\right)\phi_r(y') \left(\int_{\mathcal{X}_{-B}} \left[p_{\bX Y}(\bx',y')-p_{\bX}(\bx')p_{Y}(y')\right] d\bx'_{-B}\right)d\bx'_Bdy' \\
&=& \int_{\mathcal{X}_B} \int_{\mathcal{Y}}\prod_{l\in B} \left(1+k_l(x_l,x_l')\right)\phi_r(y') \left[p_{\bX_B Y}(\bx_B',y')-p_{\bX_B}(\bx_B')p_{Y}(y')\right]d\bx'_Bdy'\\
&=& \int_{\mathcal{X}_B} \int_{\mathcal{Y}} k_B(\bx_B,\bx_B')\phi_r(y') \left[p_{\bX_B Y}(\bx_B',y')-p_{\bX_B}(\bx_B')p_{Y}(y')\right]d\bx'_Bdy'
\end{eqnarray*}
and its norm then equals
\begin{eqnarray*}
\Vert P_{-B}(g^{[r]}) \Vert_{\F}^2 &=& \int_{\mathcal{X}_B\times\mathcal{X}_B}\int_{\mathcal{Y}\times\mathcal{Y}} k_{B}(\bx_B,\bx_B')\phi_r(y)\phi_r(y') \left[p_{\bX_B Y}(\bx_B,y)-p_{\bX_B}(\bx_B)p_{Y}(y)\right]\\
&&\left[p_{\bX_B Y}(\bx_B',y')-p_{\bX_B}(\bx_B')p_{Y}(y')\right]d\bx_B d\bx_B' dy dy'.\\
\end{eqnarray*}
Finally, we have
\begin{eqnarray*}
\HS(\bX,Y) &=& \sum_{r=1}^{\infty} \Vert g^{[r]}\Vert_{\F}^2\\
&=& \sum_{A\subseteq\Pd} \sum_{r=1}^{\infty} \Vert g_A^{[r]}\Vert_{\F}^2\\
&=& \sum_{A\subseteq\Pd} \sum_{B\subseteq A} (-1)^{\vert A\vert - \vert B\vert} \sum_{r=1}^{\infty} \Vert P_{-B}(g^{[r]}) \Vert_{\F}^2
\end{eqnarray*}
and the proof follows from
\begin{eqnarray*}
\sum_{r=1}^{\infty} \Vert P_{-B}(g^{[r]}) \Vert_{\F}^2 &=& \sum_{r=1}^{\infty} \int_{\mathcal{X}_B\times\mathcal{X}_B}\int_{\mathcal{Y}\times\mathcal{Y}} k_{B}(\bx_B,\bx_B')\phi_r(y)\phi_r(y') \left[p_{\bX_B Y}(\bx_B,y)-p_{\bX_B}(\bx_B)p_{Y}(y)\right]\\
&&\left[p_{\bX_B Y}(\bx_B',y')-p_{\bX_B}(\bx_B')p_{Y}(y')\right]d\bx_B d\bx_B' dy dy'\\
&=&  \int_{\mathcal{X}_B\times\mathcal{X}_B}\int_{\mathcal{Y}\times\mathcal{Y}} k_{B}(\bx_B,\bx_B')\left(\sum_{r=1}^{\infty}\phi_r(y)\phi_r(y')\right)\left[p_{\bX_B Y}(\bx_B,y)-p_{\bX_B}(\bx_B)p_{Y}(y)\right]\\
&&\left[p_{\bX_B Y}(\bx_B',y')-p_{\bX_B}(\bx_B')p_{Y}(y')\right]d\bx_B d\bx_B' dy dy'\\
&=& \int_{\mathcal{X}_B\times\mathcal{X}_B}\int_{\mathcal{Y}\times\mathcal{Y}} k_{B}(\bx_B,\bx_B')k_{\mathcal{Y}}(y,y')\left[p_{\bX_B Y}(\bx_B,y)-p_{\bX_B}(\bx_B)p_{Y}(y)\right]\\
&&\left[p_{\bX_B Y}(\bx_B',y')-p_{\bX_B}(\bx_B')p_{Y}(y')\right]d\bx_B d\bx_B' dy dy'\\
&=& \HS(\bX_B,Y).
\end{eqnarray*}
\end{proof}

\subsection{Proof of Proposition \ref{prop:asympHSIC}} \label{sec:asympHSIC}

\begin{proof}
We begin with the integral formulation of HSIC as in Eq. (\ref{eq:inthsic}) and plug the kernel defined in Eq. (\ref{eq:limkernel}):
\begin{eqnarray*}
\HS(\bX_A,Y)&=&\int_{\mathcal{X}_A\times\mathcal{X}_A}\int_{\mathcal{Y}\times\mathcal{Y}} k_{A}(\bx_A,\bx_A')k_{\mathcal{Y}}(y,y') \left[p_{\bX_A Y}(\bx_A,y)-p_{\bX_A}(\bx_A)p_{Y}(y)\right] \\
&&\left[p_{\bX_A Y}(\bx_A',y')-p_{\bX_A}(\bx_A')p_{Y}(y')\right]d\bx_A d\bx_A' dy dy'\\
&=& \int_{\mathcal{X}_A\times\mathcal{X}_A}\int_{\mathcal{Y}\times\mathcal{Y}} \frac{1}{\sqrt{p_{\bX_A}(\bx_A)}\sqrt{p_{\bX_A}(\bx_A')}}\prod_{l\in A} \frac{1}{h} K\left(\frac{x_l-x_l'}{h}\right) k_{\mathcal{Y}}(y,y') \\
&&\left[p_{\bX_A Y}(\bx_A,y)-p_{\bX_A}(\bx_A)p_{Y}(y)\right] \left[p_{\bX_A Y}(\bx_A',y')-p_{\bX_A}(\bx_A')p_{Y}(y')\right]d\bx_A d\bx_A' dy dy'.\\
\end{eqnarray*}
We then use a change of variables $u_l = (x_l-x_l')/h$, which leads to
\begin{eqnarray*}
\HS(\bX_A,Y)&=& \int_{\mathcal{X}_A\times\mathcal{X}_A}\int_{\mathcal{Y}\times\mathcal{Y}} \frac{1}{\sqrt{p_{\bX_A}(\bx_A)}\sqrt{p_{\bX_A}(\bx_A-h\bu_A)}}\prod_{l\in A} K\left(u_l\right) k_{\mathcal{Y}}(y,y') \\
&&\left[p_{\bX_A Y}(\bx_A,y)-p_{\bX_A}(\bx_A)p_{Y}(y)\right] \\
&&\left[p_{\bX_A Y}(\bx_A-h\bu_A,y')-p_{\bX_A}(\bx_A-h\bu_A)p_{Y}(y')\right]d\bx_A d\bu_A dy dy'.\\
\end{eqnarray*}
Now we let $h\rightarrow 0$:
\begin{eqnarray*}
\lim_{h\rightarrow 0}\HS(\bX_A,Y)&=& \int_{\mathcal{X}_A\times\mathcal{X}_A}\int_{\mathcal{Y}\times\mathcal{Y}} \frac{1}{\sqrt{p_{\bX_A}(\bx_A)}\sqrt{p_{\bX_A}(\bx_A)}}\prod_{l\in A} K\left(u_l\right) k_{\mathcal{Y}}(y,y') \\
&&\left[p_{\bX_A Y}(\bx_A,y)-p_{\bX_A}(\bx_A)p_{Y}(y)\right] \\
&&\left[p_{\bX_A Y}(\bx_A,y')-p_{\bX_A}(\bx_A)p_{Y}(y')\right]d\bx_A d\bu_A dy dy'\\
&=& \int_{\mathcal{X}_A} \prod_{l\in A} K\left(u_l\right) d\bu_A \int_{\mathcal{X}_A}\int_{\mathcal{Y}\times\mathcal{Y}} \frac{1}{p_{\bX_A}(\bx_A)}k_{\mathcal{Y}}(y,y')\\
&& \left[p_{\bX_A Y}(\bx_A,y)-p_{\bX_A}(\bx_A)p_{Y}(y)\right] \left[p_{\bX_A Y}(\bx_A,y')-p_{\bX_A}(\bx_A)p_{Y}(y')\right] d\bx_A  dy dy'\\
&=& \int_{\mathcal{X}_A}\int_{\mathcal{Y}\times\mathcal{Y}} k_{\mathcal{Y}}(y,y') \left[p_{Y\vert \bX_A =\bx_A}(y)-p_{Y}(y)\right] \\
&& \left[p_{Y\vert \bX_A =\bx_A}(y')-p_{Y}(y')\right]  p_{\bX_A}(\bx_A) d\bx_A  dy dy'\\
\end{eqnarray*}
where we have used $\int_uK(u)du=1$. Proposition  \ref{prop:asympHSIC} then follows by noting that the last equation equals $\EE_{{\bX}_A}\left(\M^2(\textup{P}_{Y},\textup{P}_{Y \vert {\bX}_A})\right)$ thanks to the integral formulation of the MMD.

\end{proof}

\subsection{Proof of Proposition \ref{prop:rank}} \label{sec:rank}

\begin{proof}
Assuming Mercer's theorem holds, we have $k(y,y') = \sum_{r=1}^{\infty} \phi_r(y)\phi_r(y')$ and
\begin{eqnarray}
\EE\left(\chi_n\right) &=& \sum_{r=1}^\infty \frac{1}{n} \sum_{i=1}^n \EE \left[\phi_r\left(y^{(i)}\right)\phi_r\left(y^{(\sigma_n((i))}\right)\right]\label{eq:step1} \\
&=& \sum_{r=1}^\infty \EE \left[\phi_r\left(y^{(1)}\right)\phi_r\left(y^{(\sigma_n((1))}\right)\right]\nonumber\\
&\rightarrow&  \sum_{r=1}^\infty \EE \left[\EE\left[\phi_r\left(Y\right)\vert V\right]\EE\left[\phi_r\left(Y\right)\vert V\right]\right]\label{eq:step2}\\
&=& \sum_{r=1}^\infty \EE \left[\EE\left[\phi_r\left(Y\right)\vert V\right]^2\right]\nonumber\\
&=& \EE_{V} \EE_{\xi,\xi'\sim \textup{P}_{Y\vert V}}k_{\mathcal{Y}}(\xi,\xi') \label{eq:step3}
\end{eqnarray}
\end{proof}
where (\ref{eq:step1}) is obtained by the absolute convergence of Mercer's series, (\ref{eq:step2}) by applying Eq. (34) in the proof of Proposition 3.2 from \cite{gamboa20} to $f=g=\phi_r$ (which is bounded since $k$ is bounded) and the absolute convergence of Mercer's series and (\ref{eq:step3}) with Eq. \ref{eq:MMDindexgroup}. The Mac Diarmid's concentration inequality given in Theorem A.1 in the proof of Proposition 3.2 from \cite{gamboa20} is unchanged and concludes the proof.

\subsection{Proof of Lemma \ref{lemma:mmshap}} \label{sec:mmshap}

\begin{proof}
We follow closely the proof of Theorem 1 from \cite{song16}. We only need to prove that for a subset $A\subseteq\Pd$ such that $l\notin A$, then 
\begin{equation}
\label{eq:val}
\textup{val}(A\cup \{l\}) - \textup{val}(A) =  \textup{val}'(B\cup \{l\}) - \textup{val}'(B)
\end{equation}
where $B=\Pd\backslash (A\cup \{l\})$. We first need the generalized law of total variance for $\M^2_{\textup{tot}}$ from Proposition \ref{prop:lawtotalvar}:
\begin{equation*}
\M^2_{\textup{tot}} = \EE_{{\bX}_A}\left[\EE_{\xi\sim\textup{P}_{Y \vert {\bX}_A}} k_{\mathcal{Y}}(\xi,\xi) - \EE_{\xi,\xi'\sim\textup{P}_{Y \vert {\bX}_A}} k_{\mathcal{Y}}(\xi,\xi')\right] + \EE_{{\bX}_A}\left(\M^2(\textup{P}_{Y},\textup{P}_{Y \vert {\bX}_A})\right).
\end{equation*}
Now we prove the equality (\ref{eq:val}) for $\textup{val}$ and $\textup{val}'$ defined in Lemma \ref{lemma:mmshap}, except that here we work without the denominator $\M^2_{\textup{tot}}$ for better readability (this does not change the proof since this same constant appears in both value functions).
\begin{eqnarray*}
\textup{val}(A\cup \{l\}) - \textup{val}(A) &=& \EE_{{\bX}_{A\cup\{l\}}}\left(\M^2(\textup{P}_{Y},\textup{P}_{Y \vert {\bX}_{A\cup\{l\}}})\right) - \EE_{{\bX}_A}\left(\M^2(\textup{P}_{Y},\textup{P}_{Y \vert {\bX}_A})\right)\\
&=& \left\{\M^2_{\textup{tot}} - \EE_{{\bX}_{A\cup\{l\}}}\left[\EE_{\xi\sim\textup{P}_{Y \vert {\bX}_{A\cup\{l\}}}} k_{\mathcal{Y}}(\xi,\xi) - \EE_{\xi,\xi'\sim\textup{P}_{Y \vert {\bX}_{A\cup\{l\}}}} k_{\mathcal{Y}}(\xi,\xi')\right]\right\}\\
&& - \left\{\M^2_{\textup{tot}} - \EE_{{\bX}_A}\left[\EE_{\xi\sim\textup{P}_{Y \vert {\bX}_A}} k_{\mathcal{Y}}(\xi,\xi) - \EE_{\xi,\xi'\sim\textup{P}_{Y \vert {\bX}_A}} k_{\mathcal{Y}}(\xi,\xi')\right] \right\}\\
&=& \EE_{{\bX}_A}\left[\EE_{\xi\sim\textup{P}_{Y \vert {\bX}_A}} k_{\mathcal{Y}}(\xi,\xi) - \EE_{\xi,\xi'\sim\textup{P}_{Y \vert {\bX}_A}} k_{\mathcal{Y}}(\xi,\xi')\right] \\
&& - \EE_{{\bX}_{A\cup\{l\}}}\left[\EE_{\xi\sim\textup{P}_{Y \vert {\bX}_{A\cup\{l\}}}} k_{\mathcal{Y}}(\xi,\xi) - \EE_{\xi,\xi'\sim\textup{P}_{Y \vert {\bX}_{A\cup\{l\}}}} k_{\mathcal{Y}}(\xi,\xi')\right]\\
&=& \EE_{{\bX}_{-(B\cup\{l\})}}\left[\EE_{\xi\sim\textup{P}_{Y \vert {\bX}_{-(B\cup\{l\})}}} k_{\mathcal{Y}}(\xi,\xi) - \EE_{\xi,\xi'\sim\textup{P}_{Y \vert {\bX}_{-(B\cup\{l\})}}} k_{\mathcal{Y}}(\xi,\xi')\right] \\
&& - \EE_{{\bX}_{-B}}\left[\EE_{\xi\sim\textup{P}_{Y \vert {\bX}_{-B}}} k_{\mathcal{Y}}(\xi,\xi) - \EE_{\xi,\xi'\sim\textup{P}_{Y \vert {\bX}_{-B}}} k_{\mathcal{Y}}(\xi,\xi')\right]\\
&=& \textup{val}'(B\cup \{l\}) - \textup{val}'(B)
\end{eqnarray*}
where we have used the generalized law of total variance for the second equality. The rest of the proof is identical to the one of Theorem 1 from \cite{song16}.

\end{proof}
\bibliographystyle{agsm}
\bibliography{KernelANOVA}

\end{document}